\documentclass{amsart}
\usepackage{amscd}
\usepackage{amssymb}
\usepackage{verbatim}

\input xy
\xyoption{all}

\newtheorem{maintheorem}{Theorem}

\newtheorem{theorem}{Theorem}[section]
\newtheorem{lemma}[theorem]{Lemma}
\newtheorem{proposition}[theorem]{Proposition}
\newtheorem{corollary}[theorem]{Corollary}

\theoremstyle{remark}
\newtheorem{remark}[theorem]{Remark}
\newtheorem{example}[theorem]{Example}

\theoremstyle{definition}
\newtheorem{definition}[theorem]{Definition}

\makeatletter\let\c@equation=\c@theorem\makeatother

\newcommand{\Cat}{\mathit{Cat}}
\newcommand{\co}{\colon\thinspace}
\newcommand{\cy}{\text{cy}}
\newcommand{\colim}{\operatornamewithlimits{colim}}

\newcommand{\ev}{\text{ev}}

\newcommand{\I}{\mathcal I}
\newcommand{\id}{\text{id}}

\newcommand{\hocolim}{\operatornamewithlimits{hocolim}}
\newcommand{\holim}{\operatornamewithlimits{holim}}
\newcommand{\Hom}{\operatorname{Hom}}
\newcommand{\K}{\operatorname{K}}

\renewcommand{\mod}{\mathit{mod}}

\newcommand{\Map}{\operatorname{Map}}
\newcommand{\op}{\mathrm{op}}

\newcommand{\sd}{\mathrm{sd}}
\newcommand{\Sp}{\mathit{Sp}}
\newcommand{\TC}{\operatorname{TC}}

\renewcommand{\TH}{\operatorname{TH}}
\newcommand{\TTC}{\mathbf{TC}}
\newcommand{\TR}{\operatorname{TR}}

\newcommand{\trc}{\text{trc}}

\newcommand{\xr}{\xrightarrow}
\newcommand{\xl}{\xleftarrow}
\begin{document}

\title{The cyclotomic trace for symmetric ring spectra}
\author{Christian Schlichtkrull}
\address{Department of Mathematics, University of Bergen, Johannes Brunsgate 12,
5008 Bergen, Norway} \email{krull@math.uib.no}
\date{\today}

\begin{abstract}
The purpose of this paper is to present a simple and explicit construction of the 
B\"okstedt-Hsiang-Madsen cyclotomic trace relating algebraic K-theory and 
topological cyclic homology. Our construction also incorporates Goodwillie's idea of a global cyclotomic trace. 
\end{abstract}

\maketitle

\section{Introduction}\label{introduction}
As defined by B\"okstedt-Hsiang-Madsen  \cite{BHM}, the cyclotomic trace 
\[
\trc\co \K(A)\to \TC(A)
\]
is a natural map relating the algebraic K-theory spectrum $\K(A)$ and the topological cyclic homology spectrum $\TC(A)$ for any connective symmetric ring spectrum $A$.
 The purpose of this paper is to present a simplified construction of this map which at the same time incorporates Goodwillie's idea of a global cyclotomic trace.  We begin by recalling the basic ingredients. 
\subsection{Topological cyclic homology}\label{TCsection}
 The definition of $\TC(A)$ is based on the model of the topological Hochschild homology spectrum $\TH(A)$ introduced by B\"okstedt ~\cite{B}. Being the realization of a cyclic spectrum this has a canonical action of the circle group $\mathbb T$ and by restriction an action of each of the finite cyclic groups $C_r$. The fixed point spectra are related by two types of structure maps
 \[
 F_r, R_r\co \TH(A)^{C_{nr}}\to \TH(A)^{C_n},
 \]
 called respectively the Frobenius and the restrictions maps. Here the Frobenius maps are the natural inclusions whereas the definition of the restriction maps depends on the 
 cyclotomic structure of $\TH(A)$. The terminology was introduced by Hesselholt-Madsen \cite{HM} and is motivated by the relationship to the theory of Witt vectors: if $A$ is commutative, then $\pi_0\TH(A)^{C_n}$ is isomorphic to the ring of truncated Witt vectors 
 $\mathbb W_{\langle n\rangle}(\pi_0(A))$ and the maps $F_r$ and $R_r$ respectively induce the Frobenius and restriction homomorphisms of Witt vectors under this isomorphism.  
Let $\mathbb I$ be the category with objects the natural numbers $n\geq 1$ and two types of morphisms, $F_r,R_r\co nr\to n$, subject to the relations
\begin{equation}\label{Icategory}
F_1=R_1=\text{id},\quad F_rF_s=F_{rs},\quad R_rR_s=R_{rs}, \quad F_rR_s=R_sF_r.
\end{equation}
Thus, any morphism in $\mathbb I$ can be written uniquely in the form $F_rR_s$. Given a prime number $p$, let $\mathbb I_p$ be the the full subcategory whose objects are the $p$-powers $p^n$. The correspondence $n\mapsto \TH(A)^{C_n}$ defines an $\mathbb I$-diagram and following \cite{BHM} we define
$$
\TC(A)=\holim_{\mathbb I}\TH(A)^{C_n}\quad\text{and}\quad 
\TC(A,p)=\holim_{\mathbb I_p}\TH(A)^{C_{p^n}}.
$$ 
Identifying $\mathbb T$ with $\mathbb T/C_n$ in the canonical way, $z\mapsto \sqrt[n]{z}$, each of the fixed point spectra $\TH(A)^{C_n}$ inherits a $\mathbb T$-action and it is natural to build this into the construction. The $\mathbb I$-diagram defining $\TC(A)$ is not a diagram of spectra with $\mathbb T$-action, but Goodwillie observes in \cite{G} that $\mathbb I$ and $\mathbb T$ can be combined into a certain twisted product category $\mathbb I\ltimes \mathbb T$ such that the correspondence $n\mapsto \TH(A)^{C_n}$ extends to an $\mathbb I\ltimes\mathbb T$-diagram which for each $n$ codifies the $\mathbb T$-action on $\TH(A)^{C_n}$. Let us write $\TTC(A)$ and $\TTC(A,p)$ for the homotopy limits over $\mathbb I\ltimes \mathbb T$ and $\mathbb I_p\ltimes\mathbb T$. There is a diagram of inclusions of categories 
\begin{equation}\label{subcategories}
\begin{CD}
\mathbb I\ltimes\mathbb T@<<< \mathbb I_p\ltimes \mathbb T@<<< \mathbb T\\
@AAA @AAA @AAA \\
\mathbb I @<<< \mathbb I_p@<<< \{1\}
\end{CD}
\end{equation}
and a corresponding diagram of homotopy limits
\[
\begin{CD}
\TTC(A)@>>>\TTC(A,p)@>>>\TH(A)^{h\mathbb T}\\
@VVV @VVV @VVV\\
\TC(A)@>>> \TC(A,p) @>>>\TH(A).
\end{CD}
\]
Here $\TH(A)^{h\mathbb T}$ denotes the homotopy fixed points of the $\mathbb T$-action on 
$\TH(A)$. Goodwillie proves in \cite{G} that the map $\TTC(A)\to \TC(A)$ becomes an equivalence after profinite completion and that the map $\TC(A)\to\TC(A,p)$ becomes an equivalence after $p$-completion. Moreover, it follows from \cite{G} that the functor $\TTC(A)$ is determined by $\TC(A)$ and $\TH(A)^{h\mathbb T}$ in the sense that there is a homotopy cartesian diagram
\[
\begin{CD}
\TTC(A)@>>> \TH(A)^{h\mathbb T}\\
@VVV @VVV\\
\TTC(A)^{\wedge}@>>> (\TH(A)^{h\mathbb T})^{\wedge}
\end{CD}
\]
where $(-)^{\wedge}$ denotes profinite completion. The cyclotomic trace lifts to a map
\[
\trc\co \K(A)\to \TTC(A)
\]
which Goodwillie calls the \emph{global cyclotomic trace}. 
Thus, composing with the map to $\TC(A)$ we recover the cyclotomic trace of 
B\"okstedt-Hsiang-Madsen while composing with the map to $\TH(A)^{h\mathbb T}$ we get the topological analogue of the Chern character with values in negative cyclic homology. The main interest in the global cyclotomic trace comes from the fact that it leads to the following integral version of the Dundas-McCarthy Theorem: if 
$A\to B$ is a map of connective symmetric ring spectra such that the induced map 
$\pi_0(A)\to \pi_0(B)$ is a surjection with nilpotent kernel, then the diagram
\[
\begin{CD}
\K(A)@>\trc>> \TTC(A)\\
@VVV @VVV\\
\K(B)@>\trc >> \TTC(B)
\end{CD}
\]
is homotopy cartesian. This is proved in \cite{DGM} and is a sharpening of earlier theorems by McCarthy \cite{McC} (for simplicial rings) and Dundas \cite{D} which state that the analogous diagram for $\TC(A)$ becomes homotopy cartesian after profinite completion.  Using the Dundas-Goodwillie-McCarthy Theorems, calculations in algebraic K-theory can often be reduced to calculations of the more accessible functor $\TC(A)$. 
We refer the reader to the excellent survey papers \cite{M} and \cite{H} for an introduction to the calculational results that can be obtained by these methods. Here we shall mainly be concerned with the technical details involved in the definition of the cyclotomic trace itself. We now give a detailed outline of the construction, followed by a discussion of how our definitions relate to those in the literature.

\subsection{The cyclotomic trace}
For simplicity we shall only consider the algebraic K-theory of free modules as opposed to the topological version of ``finitely generated projective" modules; as in the case of ordinary rings it follows from a cofinality argument that the resulting algebraic K-theories only differ in degree zero. Thus, given a connective symmetric ring spectrum $A$, let $\mathcal F_A$ be the category of finitely generated free left $A$-modules of the standard form $A^{\vee n}$. This is a  \emph{spectral category} (a category enriched in symmetric spectra) in the sense that there is a symmetric spectrum of ``morphisms'' relating any two objects. We shall define an associated topological category $\omega\mathcal F_A$ of ``weak equivalences'' that has the same objects and whose morphism spaces may be identified with the spaces of stable equivalences  between the objects in $\mathcal F_A$. The symmetric monoidal structure of $\mathcal F_A$ makes the classifying space $B(\omega\mathcal F_A)$ the underlying space of a $\Gamma$-space in the sense of Segal \cite{S} and the algebraic $\K$-theory spectrum $\K(A)$ is the associated spectrum. Applying Waldhausen's cyclic classifying space construction we similarly get a 
$\Gamma$-space $B^{\cy}(\omega\mathcal F_A)$ whose associated spectrum is the cyclic algebraic K-theory spectrum $\K^{\cy}(A)$. 

One can also evaluate the cyclic bar construction on the spectral category $\mathcal F_A$ itself and we write $\TH(\mathcal F_A)$ for the Dundas-McCarthy model of the topological Hochschild homology spectrum. B\"okstedt's model $\TH(A)$ is obtained by restricting to the subcategory of $\mathcal F_A$ containing only the object $A$ and it is proved in \cite{DM} that the inclusion induces an equivalence  $\TH(A)\to\TH(\mathcal F_A)$ of spectra with cyclotomic structure. It follows that there are induced equivalences of the fixed point spectra and the homotopy limits defining the various form of topological cyclic homology. The advantage of 
$\TH(\mathcal F_A)$ is that the symmetric monoidal structure of $\mathcal F_A$ gives rise to an extra spectrum coordinate (making $\TH(\mathcal F_A)$ a symmetric bispectrum) which is compatible with the spectrum structure of $\K(A)$. Thus, there is a canonical map
\[
\K^{\cy}(A)\to \TH(\mathcal F_A)
\]
which is essentially obtained by including the spaces of stable equivalences in the full morphism spaces of maps between the objects in $\mathcal F_A$. This is in fact a map of spectra with cyclotomic structure and exploiting this we get a map
\[
\K^{\cy}(A)\to \TR(\mathcal F_A)=\holim_{R_r}\TH(\mathcal F_A)^{C_r},
\]  
where by definition $\TR(\mathcal F_A)$ is the homotopy limit over the restriction maps. For commutative $A$ it follows from \cite{HM} that $\pi_0\TR(\mathcal F_A)$ is isomorphic to the ring of big Witt vectors $1+t\bar A[[t]]$ on the ring $\bar A=\pi_0(A)$ and it is natural to view 
$\TR(\mathcal F_A)$ as a topological refinement of the Witt vector construction. From this point of view, the following example makes it natural to view  the above map as a kind of ``characteristic polynomial'' (although for this interpretation one may argue that our definition of 
$\K^{\cy}(A)$ is not optimal). 

\begin{example}\label{Autexample}
It is illuminating to consider the case where $A$ is the Eilenberg-Mac Lane spectrum of an ordinary commutative ring $\bar A$. The cyclic algebraic K-theory $\K^{\cy}(A)$ may then be identified with the algebraic K-theory of the automorphism category $\operatorname{Aut}(\bar A)$, thought of as a category with coproducs. An object of this category is a pair $(P,\alpha)$ given by a finitely generated free $A$-module $P$ and an automorphism $\alpha$ of $P$. 
In this case the above map induces the characteristic polynomial on $\pi_0$ in the sense that an object $(P,\alpha)$ is mapped to
\[
\det(1-t\alpha)\in \pi_0\TR(\mathcal F_A)=\{1+t\bar A[[t]]\}.
\]
Redefining $\K^{\cy}(A)$ by applying Waldhausen's $S_{\bullet}$-construction instead of Segal's $\Gamma$-space approach gives a spectrum that may be identified with the algebraic K-theory of $\operatorname{Aut}(\bar A)$, thought of as an exact category in the usual way.
\end{example}   

 Let $N$ be the multiplicative monoid of natural numbers and write $N\ltimes\mathbb T$ for the semidirect product with $N$ acting from the right on $\mathbb T$ through the power maps. It follows formally from the definition of the category $\mathbb I\ltimes  \mathbb T$ that the topological monoid $N\ltimes\mathbb T$ acts on $\TR(\mathcal F_A)$ and that 
 $\TTC(\mathcal F_A)$ can be identified with the homotopy fixed point spectrum $\TR(\mathcal F_A)^{h(N\ltimes\mathbb T)}$. The above ``characteristic polynomial'' is $N\ltimes\mathbb T$-equivariant and induces a map of homotopy fixed points
\[
\K^{\cy}(A)^{h(N\ltimes\mathbb T)}\to \TR(\mathcal F_A)^{h(N\ltimes \mathbb T)}=
\TTC(\mathcal F_A).
\] 
Thus, in order to define the cyclotomic trace it remains to map the spectrum $\K(A)$ into the homotopy fixed points of $\K^{\cy}(A)$.  With this in mind we prove the following general result.
\begin{maintheorem}\label{Bcytheorem}
Let $\mathcal C$ be a small topological category. Then $N\ltimes \mathbb T$ acts on 
$B^{\cy}(\mathcal C)$ and there is a natural map 
\[
B^{\cy}(\mathcal C)^{h(N\ltimes \mathbb T)}\to \Map(BN,B(\mathcal C)) 
\]
which is a weak homotopy equivalence if $\mathcal C$ is groupoid-like (that is, if the component category $\pi_0\mathcal C$ is a groupoid). 
\end{maintheorem}

Applying this to $w\mathcal F_A$ we define a 
$\Gamma$-space $B'(w\mathcal F_A)$ by forming the homotopy pullback of the diagram
\[
B^{\cy}(w\mathcal F_A)^{h(N\ltimes\mathbb T)}\xr{\sim}
\Map(BN,B(w\mathcal F_A))\xl{}
B(w\mathcal F_A)
\] 
where the right hand map is defined by including $B(w\mathcal F_A)$ as the constant functions.
The associated spectrum $\K'(A)$ is canonically equivalent to $\K(A)$ and maps naturally to the homotopy fixed points of $\K^{\cy}(A)$. Summarizing, our definition of the cyclotomic trace is represented by the chain of maps
\[
\trc\co\K(A)\xl{\sim}\K'(A)\to \K^{\cy}(A)^{h(N\ltimes\mathbb T)}\to 
\TR(\mathcal F_A)^{h(N\ltimes \mathbb T)} =\TTC(\mathcal F_A)\simeq
\TTC(A).
\]
In situations where it is important to have a direct map relating algebraic K-theory and topological cyclic homology we may of course choose to work with the models $\K'(A)$ and 
$\TTC(\mathcal F_A)$. 
It is worth noting that 
$\Map(BN,B(w\mathcal F_A))$ is the inverse limit of a diagram of fibrations  
\[
P\mapsto \Map(\textstyle\prod_{p\in P}B\langle p\rangle, B(w\mathcal F_A))
\]
where $P$ runs through the finite sets of prime numbers and $\langle p\rangle$ denotes the 
multiplicative monoid generated by $p$ (thus, the domain is homotopy equivalent to a $|P|$-dimensional torus). The projection onto 
$\Map(B\langle p\rangle, B(w\mathcal F_A))$ corresponds via the cyclotomic trace to the projection of $\TTC(A)$ onto $\TTC(A,p)$.
   
There are two main innovations in the approach to the cyclotomic trace taken here. 
The first is that the $\Gamma$-space structures we use to define the spectra $\K(A)$ and 
$\TH(\mathcal F_A)$ are considerably simpler than those considered in \cite{BHM}. The second is that we avoid ``inverting the weak equivalences'' in $w\mathcal F_A$ before mapping into $\TH(\mathcal F_A)$: the method for comparing the bar construction to the cyclic bar construction used in \cite{BHM} and \cite{G} involves replacing a grouplike monoid by an equivalent group and this procedure was refined by Dundas \cite{Dloc}, \cite{D2}, to a localization functor on the categorical level.  Whereas this procedure works fine for many purposes it does not behave well with respect to multiplicative structures. Thus, even though $\K(A)$ and $\TTC(A)$ are $E_{\infty}$ ring spectra if $A$ is commutative, it is not obvious how to make the cyclotomic trace an $E_{\infty}$ map from this point of view. In our formulation we avoid inverting the weak equivalences by directly analyzing the homotopy fixed points of the cyclic bar construction and all the steps in the construction presented here are compatible with products. Based on this we show in 
\cite{SchEinfty} how to refine the cyclotomic trace to an $E_{\infty}$ map.  

\subsection{Variations and generalizations}
In writing this paper, the main priority has been to keep the constructions as simple and explicit as possible. We here list a number of possible variations and generalizations. First of all, we have chosen to work with symmetric spectra of topological spaces, but one could have worked with symmetric spectra of simplicial sets throughout. This would in fact have simplified some of the arguments since then we would not have to worry about the symmetric spectra being 
``well-based''. Secondly, while our construction of the algebraic K-theory spectrum is based on Segal's $\Gamma$-space approach, we could have  chosen to use Waldhausen's $S_{\bullet}$ construction instead, see \cite{W1}. This would give an equivalent model of the algebraic K-theory spectrum, but for the cyclic algebraic K-theory spectrum one would get a new and, arguably, better behaved theory, cf.\  Example~ \ref{Autexample}. On the other hand, we find the simplicity of the $\Gamma$-space construction very appealing and this model is convenient when making the cyclotomic trace an $E_{\infty}$ map \cite{SchEinfty}. We also remark that the constructions in this paper can be used to define the cyclotomic trace for more general symmetric monoidal spectral categories along the lines of \cite{D1}. It remains an interesting question how to define a good version of the cyclotomic trace for symmetric ring spectra that are not connective.   

We have aimed at making the paper reasonable self contained and we have tried to give suitable references along the way and to explain how our definitions compare to earlier ones.
We have been particularly influenced by the papers by B\"okstedt-Hsiang-Madsen \cite{BHM}, Goodwillie \cite{G}, Hesselholt-Madsen \cite{HM}, 
Dundas-McCarthy \cite{DM}, and Dundas \cite{D1,D2}.

\subsection*{Organization of the paper}
We begin in Section \ref{holimsection} by fixing notation for symmetric spectra and homotopy
limits. Here we also include a detailed discussion of the homotopy limit of a diagram indexed by a Grothendieck construction. This material is used in later sections when analyzing homotopy limits of diagrams indexed by the categories $\mathbb I$ and $\mathbb I\ltimes \mathbb T$. 
The study of algebraic K-theory begins in Section~\ref{algebraicKsection} where we introduce the category $w\mathcal F_A$ of stable equivalences and the associated algebraic K-theory spectrum $\K(A)$. In Section 
\ref{cyclicalgebraicKsection} we consider the cyclic analogue $\K^{\cy}(A)$ and based on Theorem \ref{Bcytheorem} we show how to relate $\K(A)$ to the homotopy fixed points of the latter. The definition of the topological cyclic homology spectrum is then recalled in Section \ref{cyclotomictracesection} where we define the cyclotomic trace. Finally, we analyze the homotopy fixed points of the cyclic bar construction and prove Theorem \ref{Bcytheorem} in Section 
\ref{hofixsection}. 

\section{Symmetric spectra and homotopy limits}\label{holimsection}
In this section we first fix notation for symmetric spectra and homotopy limits. We then give a detailed account of the dual Grothendieck construction and the dual version of Thomason's homotopy colimit theorem which describes the homotopy limit of a diagram indexed by a Grothendieck category. The reason for including this material is that the general theory specializes to a canonical approach for analyzing homotopy limits of diagrams indexed by the categories $\mathbb I$ and $\mathbb I\ltimes \mathbb T$ entering in the definition of the cyclotomic trace. 

\subsection{Symmetric spectra}\label{symmetricsection}
We work in the categories $\mathcal U$ and $\mathcal T$ of unbased and based compactly generated weak Hausdorff spaces. By a spectrum we understand a sequence of based spaces 
$E(n)$ for $n\geq 0$, together with a sequence of based structure maps 
$\sigma\co S^1\wedge E(n)\to E(1+n)$.  A symmetric spectrum is a spectrum in which the $n$th space 
$E(n)$ comes equipped with a base point preserving action of the symmetric group $\Sigma_n$ such that the iterated structure maps 
\[
\sigma^m\co S^m\wedge E(n)\to E(m+n) 
\] 
are $\Sigma_m\times \Sigma_n$-equivariant. We write $\Sp^{\Sigma}$ for the topological category of symmetric spectra in which a morphism $E\to E'$ is a sequence of $\Sigma_n$-equivariant based maps $E(n)\to E'(n)$ that commute with the structure maps. The morphisms space $\Map_{\Sp^{\Sigma}}(E,E')$ is topologized as a subspace of the product of the based mapping spaces $\Map(E(n),E'(n))$. It is proved in \cite{HSS} (in the simplicial setting) and \cite{MMSS} that $\Sp^{\Sigma}$ has a stable model category structure that makes it Quillen equivalent to the usual category of spectra. In this 
model structure a symmetric spectrum $E$ is fibrant if and only if it is an $\Omega$-spectrum in the sense that the adjoint structure maps $E(n)\to \Omega E(n+1)$ are weak homotopy equivalences for $n\geq 0$. A symmetric spectrum is said to be a positive $\Omega$-spectrum if the adjoint structure maps are weak homotopy equivalences for $n\geq 1$.
We say that a map of symmetric spectra $E\to E'$ is a $\pi_*$-isomorphism if it induces an isomorphism on spectrum homotopy groups. A symmetric spectrum $E$ is \emph{semistable} if there exists an $\Omega$-spectrum $E'$ and a $\pi_*$-isomorphism $E\to E'$. Choosing as fibrant replacement in the stable model structure one can always find an $\Omega$-spectrum $E'$ and a map $E\to E'$ which is a stable equivalence (a weak equivalence in the stable model structure) but the point is that a stable equivalence need not be a $\pi_*$-isomorphism. In fact, there is an obvious candidate for such a fibrant replacement as we now recall. For each $m\geq 0$, we define the shifted 
spectrum $E[m]$ to be the symmetric spectrum with $n$th space $E(m+n)$ and structure maps
\[
S^1\wedge E(m+n)\xr{\sigma} E(1+m+n)\xr{\tau_{1,m}\sqcup 1_n}E(m+1+n)
\]  
where $\tau_{1,m}\sqcup 1_n$ is the permutation that acts by the $(1,m)$-shuffle $\tau_{1,m}$ 
on the first 
$1+m$ elements and is the identity on the last $n$ elements. The group $\Sigma_n$ acts on $E[m](n)$ via the inclusion $\Sigma_n\to \Sigma_{m+n}$ that maps an element $\alpha$ in $\Sigma_n$ to the permutation that acts as the identity on the first $m$ elements and as $\alpha$ on the last 
$n$ elements. Following \cite{HSS} we write $RE=\Omega E[1]$ and consider the map of symmetric spectra $\tilde\sigma\co E\to RE$ which in spectrum degree $n$ is the adjoint structure map 
$E(n)\to \Omega E(1+n)$. Let $R^{\infty}E$ be the homotopy colimit (or telescope) of the sequence of symmetric spectra $R^mE$ under the maps 
$R^m(\tilde\sigma)\co R^mE\to R^{m+1}E$. Thus, identifying $R^mE$ with the symmetric spectrum $\Omega^mE[m]$, the map $R^m(\tilde\sigma)$ is given in spectrum degree $n$ by
\begin{align*}
\Map(S^m,E(m+n))&\to \Map(S^1\wedge S^m,S^1\wedge E(m+n))\\
&\to \Map(S^{1+m},E(1+m+n)),
\end{align*}
where the first arrow takes a based map $f$ to 
$\id_{S^1}\wedge f$, and the second arrow is induced by the structure map of $E$. 
The inclusion of $E$ as the first term of the system defines a map $E\to R^{\infty}E$ and it follows from \cite{HSS}, Proposition~5.6.2, that $E$ is semistable if and only if $R^{\infty}E$ is an  $\Omega$-spectrum and this map is a $\pi_*$-isomorphism. Since a stable equivalence between $\Omega$-spectra is a level-equivalence it follows that a map between semistable symmetric spectra is a stable equivalence if and only if it is a $\pi_*$-isomorphism. 
The class of semistable symmetric spectra includes the (positive) symmetric 
$\Omega$-spectra and more generally any symmetric spectrum whose homotopy groups stabilize in the sense that the homomorphisms 
in the systems defining the spectrum homotopy groups eventually become isomorphisms. For instance, this includes the suspension spectra.  

\subsection{Homotopy limits}\label{holimDefsection}
We shall follow Bousfield-Kan \cite{BK} in the definition of homotopy limits except that we work topologically instead of simplicially. Thus, let $\mathcal K$ be a small topological category (a small category enriched in $\mathcal U$) by which we mean that the morphism sets 
$\mathcal K(K,K')$ are topologized and that composition is continuous. We shall tacitly assume throughout that a topological category is well-based in the sense that the identity morphisms provide each of the morphism spaces $\mathcal K(K,K)$ with a non-degenerate base point. Given an object $K$ in $\mathcal K$ we follow \cite{MacL} and write 
$(\mathcal K\downarrow K)$ for the category of objects in $\mathcal K$ over $K$. Thus, the set of objects is topologized as the disjoint union of the morphism spaces $\mathcal K(K',K)$ where $K'$ ranges over the objects in $\mathcal K$. Taking this into account, the classifying space
$B(\mathcal K\downarrow K)$ can be identified with the realization of the simplicial space 
\[
[n]\mapsto \coprod_{K_0,\dots,K_n}\mathcal K(K_0,K)\times \mathcal K(K_1,K_0)\times
\dots\times \mathcal K(K_n,K_{n-1}),
\] 
see also \cite{HV}. 
Letting $K$ vary, the correspondence 
$K\mapsto B(\mathcal K\downarrow K)$ defines a diagram of spaces which we shall denote by 
$B(\mathcal K\downarrow-)$. The homotopy limit of a diagram $X\co \mathcal K\to \mathcal U$ is defined to be the space of natural maps of $\mathcal K$-diagrams
\[
\holim_{\mathcal K}X=\Map_{\mathcal K}(B(\mathcal K\downarrow-),X),
\] 
topologized as a subspace of the product of the spaces 
$\Map(B(\mathcal K\downarrow K),X(K))$. Notice, that if $X$ is a diagram of based spaces, then $\holim_{\mathcal K}X$ is naturally a based space.
\begin{example}
Let $\mathcal K$ be the one-object category associated with a topological monoid $G$. Then a 
$\mathcal K$-diagram $X$ is the same thing as a $G$-space and $\holim_{\mathcal K}X$ is the homotopy fixed point space $X^{hG}$ defined by 
\[
X^{hG}=\Map_G(EG,X)
\] 
where $EG$ denotes the one-sided bar construction $B(G,G,*)$. 
\end{example}

We refer the reader to the papers \cite{BK}, \cite{Hi}, and  \cite{HV} for more details on homotopy limits. The main feature of the construction is that if $X\to X'$ is a map of $\mathcal K$-diagrams such that $X(K)\to X'(K)$ is a weak homotopy equivalence for each object $K$ in $\mathcal K$, then the induced map of homotopy limits is again a weak homotopy equivalence. 

Consider now a diagram of symmetric spectra $X\co \mathcal K\to \Sp^{\Sigma}$ where $\mathcal K$ is again a small topological category. Then we apply the above homotopy limit construction in each spectrum degree to get the symmetric spectrum $\holim_{\mathcal K}X$ with $n$th space $\holim_{\mathcal K}X(n)$. In this paper we shall only use this construction in the case where $X$ is a diagram of positive $\Omega$-spectra. Since a stable equivalence of positive $\Omega$-spectra is a weak homotopy equivalence in each positive spectrum degree, 
it follows that this homotopy limit functor takes level-wise stable equivalences of 
$\mathcal K$-diagrams of positive $\Omega$-spectra to stable equivalences. (In order to have a homotopically well-behaved homotopy limit functor on general diagrams one should first apply a fibrant replacement functor in $\Sp^{\Sigma}$). 

\subsection{The categorical Grothendieck construction}
For our purposes the relevant Grothendieck construction is the dual of that considered in 
\cite{T}. Thus, let $F\co \mathcal K^{\op}\to \Cat $ be a contravariant functor from a small topological category $\mathcal K$ to the category of small topological categories. The Grothendieck construction $\mathcal K\ltimes F$ is the category with objects $(K,A)$ where $K$ is an object in $\mathcal K$ and $A$ is an object of $F(K)$. A  morphism $(k,a)$ in 
$\mathcal K\ltimes F$ from $(K,A)$ to $(K',A')$ is a morphism $k\co K\to K'$ in $\mathcal K$ together with a morphism $a\co A\to F(k)(A')$ in $F(K)$. The morphism spaces are topologized in the obvious way and composition is defined by
\[
(k',a')\circ (k,a)=(k'\circ k, F(k)(a')\circ a).
\]  
The notation is motivated by the special case where $\mathcal K$ is the one-object category associated to a topological monoid $G$ and $F$ is the contravariant functor determined by a right action of $G$ on a topological monoid $H$. In this case the Grothendieck construction is the usual semidirect product $G\ltimes H$. 

For each object $K$ in $\mathcal K$ there is a canonical functor 
$i_K\co F(K)\to \mathcal K\ltimes F$ defined by mapping an object $A$ in $F(K)$ to $(K,A)$. Given a diagram of spaces $X\co \mathcal K\ltimes F\to \mathcal U$, we write $i_K^*X$  for the composition
$X\circ i_K$ and consider the associated $\mathcal K$-diagram
\begin{equation}\label{K-F(K)diagram}
K\mapsto \holim_{F(K)}i_K^*X.
\end{equation}
The structure maps are induced by the functorial properties of the homotopy limit.
The following result is essentially the dual version of Thomason's homotopy colimit theorem 
\cite{T}. We include a detailed discussion here for easy reference.

\begin{theorem}\label{holimtheorem}
Given a diagram $X\co \mathcal K\ltimes F\to \mathcal U$ there is a natural weak homotopy equivalence
\[
\lambda\co \holim_{\mathcal K\ltimes F}X\xr{\sim} \holim_{K\in \mathcal K}\holim_{F(K)}i_K^*X. 
\]
\end{theorem}
We first define the map $\lambda$ and consider the examples relevant for the cyclotomic trace. The proof will be given at the end of the section. By definition, the target is the space of natural maps
\[
\Map_{K\in\mathcal K}(B(\mathcal K\downarrow K),\Map_{F(K)}(B(F(K)\downarrow-),i_K^*X)).
\]
Thus, an element can be identified with a natural family of maps
\[
\alpha_{(K,A)}\co B(\mathcal K\downarrow K)\times B(F(K)\downarrow  A)\to X(K,A)
\]
indexed by the objects $(K,A)$ in $\mathcal K\ltimes F$. The naturality condition amounts to (i) that $\alpha_{(K,A)}$ is natural in $A$ for each fixed $K$, and (ii) that given a morphism 
$f\co K\to L$ in $\mathcal K$ and an object $A$ in $F(L)$, the diagram 
\[
\xymatrix{
 B(\mathcal K\downarrow K)\times B(F(K)\downarrow F(f)(A))\ar[rr]^-{\alpha_{(K,F(f)(A))}}
& &X(K,F(f)(A))\ar[dd]^{X(f,\id)} \\
B(\mathcal K\downarrow K)\times B(F(L)\downarrow A)\ar[u]^{\id\times F(f)} 
\ar[d]_{f_*\times \id} \\
 B(\mathcal K\downarrow L)\times B(F(L)\downarrow A)\ar[rr]^-{\alpha_{(L,A)}}& &X(L,A) 
}
\]
is commutative. Similarly, we represent an element in the domain of $\lambda$ by a natural 
family of maps 
\[
\beta_{(K,A)}\co B(\mathcal K\ltimes F\downarrow (K,A))\to X(K,A)
\]
indexed by the objects $(K,A)$ in $\mathcal K\ltimes F$. Let now the object $(K,A)$ be fixed and consider the functor 
\[
\Theta_{(K,A)}\co (\mathcal K\downarrow K)\times (F(K) \downarrow A)
\to (\mathcal K\ltimes F\downarrow (K,A))
\]
that maps a pair of objects $k\co K_0\to K$ and $a\co A_0\to A$ in the domain category 
to the object
\[
(k,F(k)(a))\co (K_0,F(k)(A_0))\to (K,A).
\]
A morphism in the domain is represented by a pair of commutative diagrams
\[
\xymatrix{
K_0\ar[rr]^{k_0}\ar[dr]_{k} & &K_0'\ar[dl]^{k'}\\
& K &
}\qquad \quad
\xymatrix{
A_0\ar[rr]^{a_0}\ar[dr]_{a} & &A_0'\ar[dl]^{a'}\\
& A &
}
\]
and this is mapped to the morphism represented by the diagram
\[
\xymatrix{
(K_0,F(k)(A_0))\ar[rr]^{(k_0,F(k)(a_0))}\ar[dr]_{(k,F(k)(a))} & &
(K_0',F(k')(A_0'))\ar[dl]^{(k',F(k')(a'))}\\
&(K,A). &
}
\] 
Since the classifying space functor preserves products there is an induced map
\[
\Theta_{(K,A)}\co B(\mathcal K\downarrow K)\times B(F(K)\downarrow A)\to 
B(\mathcal K\ltimes F\downarrow (K,A)).
\]

\begin{definition}
The map $\lambda$ in Theorem \ref{holimtheorem} 
is defined by associating to an element $\beta=\{\beta_{(K,A)}\}$ in the domain the element 
$\lambda(\beta)=\alpha$ given by
\[
\alpha_{(K,A)}\co B(\mathcal K\downarrow K)\times B(F(K)\downarrow A)\xr{\Theta_{(K,A)}} 
B(\mathcal K\ltimes F\downarrow (K,A))\xr{\beta_{(K,A)}}X(K,A). 
\]
\end{definition}

One easily verifies the required naturality conditions. The following lemma gives a convenient criterion for checking when the map $\lambda$ is in fact a homeomorphism. 
\begin{lemma}\label{lambdaiso}
Suppose that for each morphism $f\co K\to L$ in $\mathcal K$ and each object $A$ in $F(L)$ 
the functor 
\[
F(f)\co (F(L)\downarrow A)\to (F(K)\downarrow F(f)(A))
\]
is an isomorphism of categories. Then $\lambda$ is a homeomorphism. 
\end{lemma} 
\begin{proof}
The assumption in the lemma implies that the functors $\Theta_{(K,A)}$ are isomomorphisms of categories and consequently the induced maps of classifying spaces are homeomorphisms. Using this one easily defines an inverse of $\lambda$. 
\end{proof}

\begin{example}\label{semidirectexample}
Let $\mathcal K$ be the one-object category associated with a topological monoid $G$ and let 
$F$ be the functor specified by a right $G$-action on a topological monoid $H$; written 
$b\cdot a=b^a$ for $a\in G$ and $b\in H$. Then the category $\mathcal K\ltimes F$ is the one-object category associated to the semidirect product monoid $G\ltimes H$ with multiplication
\[
(a_1,b_1)\cdot(a_2,b_2)=(a_1a_2,b_1^{a_2}b_2),\qquad a_1,a_2\in G,\quad b_1, b_2\in H.
\] 
A $G\ltimes H$-action on a space $X$ amounts to a space equipped with an action of each of the monoids $H$ and $G$ such that the relation $b(ax)=a(b^ax)$ holds for all $a\in G$, $b\in H$, and $x\in X$. The monoid $G$ acts from the right on $EH$ and from the left on the homotopy fixed points $X^{hH}$ by $(a\beta)(e)=a\beta(ea)$, for $a\in G$, $\beta\in X^{hH}$, and  
$e\in EH$. An element of the fixed point space $(X^{hH})^{hG}$ can be identified with a map 
$\alpha\co EG\times EH\to X$ such that 
\[
 \alpha(ae_1,e_2)=a\alpha(e_1,e_2a),\quad \alpha(e_1,be_2)=b\alpha(e_1,e_2)
\]
for all $e_1\in EG$, $e_2\in EH$, $a\in G$, and $b\in H$. In this case the weak homotopy equivalence $\lambda\co X^{h(G\ltimes H)}\to (X^{hH})^{hG}$ is induced by the simplicial map 
\[
\Theta\co B_{\bullet}(G,G,*)\times B_{\bullet}(H,H,*)\to B_{\bullet}(G\ltimes H,G\ltimes H,*)
\]
defined by
\[
\Theta((a_0,\dots,a_n),(b_0,\dots,b_n))= 
((a_0,b_0^{a_0}),(a_1,b_1^{a_0a_1}),\dots,(a_n,b_n^{a_0\dots a_n})).
\]
Notice, that if $G$ is group, then $\lambda$ is a homeomorphism by Lemma \ref{lambdaiso}. 
\end{example}

\begin{example}\label{Iexample}
The category $\mathbb I$ from Section \ref{TCsection} can also be realized as a Grothendieck construction as we now explain. Let $N$ be the multiplicative monoid of (positive) natural numbers and let us view $N$ as a category with a single object $*$ in the usual way. We write 
$\mathcal N$ for the category $(N\downarrow *)$ such that an object in $\mathcal N$ is a natural number $n$ and a morphism $s\co m\to n$ is an element $s$ in $N$ with $m=ns$. The monoid $N$ naturally acts on $\mathcal N$ from the left and since $N$ is commutative this is also a right action. The Grothendieck construction $N\ltimes \mathcal N$ again has objects the natural numbers and a morphism $(r,s)\co m\to n$ is a pair of elements $r,s$ in $N$ such that 
$m=rns$. This category is isomorphic to $\mathbb I$ as one sees by identifying $F_r$ with 
$(r,1)$ and $R_s$ with $(1,s)$. It follows from Lemma \ref{lambdaiso} that we have a canonical homeomorphism
\[
\lambda\co \holim_{\mathbb I}X\xr{\sim}(\holim_{\mathcal N}X)^{hN}
\]    
for any $\mathbb I$-diagram $X$. This observation is originally due to Goodwillie \cite{G}. 
\end{example}

\begin{example}\label{ITexample} 
Let again $N$ be the multiplicative monoid of natural numbers and let $N$ act from the right on the circle group $\mathbb T$ via the power maps, $z\cdot r=z^r$. This induces a functor 
$\mathbb I^{\op}\to N^{\op}\to\Cat$ and the category $\mathbb I\ltimes \mathbb T$ from Section 
\ref{TCsection} is the associated Grothendieck construction. Thus, 
$\mathbb I\ltimes \mathbb T$ has objects the natural numbers and a morphism 
$(r,s,z)\co m\to n$ is a pair of elements $r,s$ in $N$ such that $m=rns$, together with an element $z$ in $\mathbb T$. We topologize the morphism sets as disjoint unions of copies of 
$\mathbb T$ and composition is defined by
\[
(r_1,s_1,z_1)\cdot (r_2,s_2,z_2)=(r_1r_2,s_1s_2,z_1^{r_2}z_2).
\] 
Notice that there are isomorphisms
\[
\mathbb I\ltimes \mathbb T\cong (N\ltimes \mathcal N)\ltimes \mathbb T
\cong N\ltimes(\mathcal N\times\mathbb T)\cong (N\ltimes\mathbb T)\ltimes \mathcal N.
\]
Given an $\mathbb I\ltimes \mathbb T$-diagram $X$ we therefore have canonical weak homotopy equivalences
\[
\holim_{\mathbb I\ltimes \mathbb T}X \xr{\sim}
(\holim_{\mathcal N}X)^{h(N\ltimes \mathbb T)}\xr{\sim}
(\holim_{\mathcal N}X^{h\mathbb T})^{hN}
\]
where in fact the first map is a homeomorphism by Lemma \ref{lambdaiso}. 
\end{example}

\subsection{The proof of Theorem \ref{holimtheorem}}
The proof follows the same outline as the proof of the dual result in \cite{T}. 
Let $p\co \mathcal K\ltimes F\to \mathcal K$ be the functor that maps an object $(K,A)$ to $K$. In order to verify that $\lambda$ is a weak homotopy equivalence we shall compare the 
$\mathcal K$-diagram (\ref{K-F(K)diagram}) to the homotopy right Kan extension of $X$ along the functor $p$, that is, to the $\mathcal K$-diagram 
\[
K\mapsto \holim_{(K\downarrow p)}\pi_K^*X
\]
where $(K\downarrow p)$ is the category with objects $(f,A)$ given by a morphism 
$f\co K\to L$ in $\mathcal K$ and an object $A$ in $F(L)$, and where 
$\pi_K\co (K\downarrow p)\to \mathcal K\ltimes F$ is the forgetful functor that forgets the morphism $f$. This is the homotopical analogue of the categorical Kan extension, see e.g.\ 
\cite{MacL}.  The functors $\pi_K$ assemble to give a map of $\mathcal K$-diagrams
\[
\holim_{\mathcal K\ltimes F}X\to \holim_{(K\downarrow p)}\pi_K^*X
\]
where we view the domain as a constant diagram. We write $\lambda_2$ for the induced map
\[
\lambda_2\co \holim_{\mathcal K\ltimes F}X\to \lim_{K\in\mathcal K}\holim_{(K\downarrow p)}
\pi_K^*X\to \holim_{K\in\mathcal K}\holim_{(K\downarrow p)}\pi_K^*X.
\]
The following lemma is standard; see e.g.\ \cite{HV}, for the dual version for homotopy colimits.
\begin{lemma}
The map $\lambda_2$ is a weak homotopy equivalence.\qed
\end{lemma}
Let again $K$ be an object in $\mathcal K$ and let $r_K\co (K\downarrow p)\to F(K)$ be the 
functor that maps an object $(f,A)$ to $F(f)(A)$. 
A morphism $(f,A)\to(f',A')$ in $(K\downarrow p)$ is a morphism 
$(l,a)\co p(f,A)\to p(f',A')$ in $\mathcal K\ltimes F$ such that $lf=f'$ and $r_K$ takes 
this to 
\[
F(f)(a)\co F(f)(A)\to F(f) F(l)(A')=F(f')(A').
\]
The composite functor $i_K\circ r_K\co (K\downarrow p)\to \mathcal K\ltimes F$ is related 
to $\pi_K$ by the natural transformation $i_K\circ r_K\to \pi_K$ which for an object 
$(f\co K\to L,A)$ in $(K\downarrow p)$ is defined by 
\[
(f,\id)\co (K,F(f)(A))\to (L,A).
\] 
This gives a map of homotopy limits for each $K$,
\[
\holim_{F(K)}i_K^*X\xr{r_K} \holim_{(K\downarrow p)}r_K^*i_K^*X\to
\holim_{(K\downarrow p)}\pi_K^*X,
\]
and one checks that this is a natural map of $\mathcal K$-diagrams.

\begin{lemma}
The induced map of homotopy limits
\[
\lambda_1\co \holim_{K\in \mathcal K} \holim_{F(K)} i_K^*X\to \holim_{K\in \mathcal K}
\holim_{(K\downarrow p)}\pi_K^*X
\]
is a weak homotopy equivalence.  
\end{lemma}
\begin{proof}
We show that the map of $\mathcal K$-diagrams defined above is in fact a weak homotopy equivalence for each $K$. The result then follows from the homotopy invariance of 
homotopy limits. Notice that the functor $r_K$ has a left adjoint $j_K\co F(K)\to 
(K\downarrow p)$ that takes an object $A$ in $F(K)$ to $(\id\co K\to K,A)$. This functor 
is a lift of $i_K$ in the sense that $\pi_K\circ j_K=i_K$. It follows that there is an induced 
map of homotopy limits 
\[
\holim_{(K\downarrow p)}\pi_K^*X\to \holim_{F(K)}j_K^*\pi_K^*X=\holim_{F(K)}i_K^*X
\]
which is a left inverse of the map in question.  Furthermore, since $j_K$ has a right adjoint it is
left homotopy cofinal in the sense that the categories $(j_K\downarrow (f,A))$ have contractible 
classifying space for each object $(f,A)$ in $(K\downarrow p)$. It therefore follows from the homotopy cofinality theorem \cite{BK}, Theorem XI.9.2, that the map of homotopy limits 
induced by $j_K$ is a weak homotopy equivalence. 
\end{proof}

Combining the above lemmas we get a chain of weak homotopy equivalences
\[
\holim_{K\in\mathcal K}\holim_{F(K)}i_K^*X\xr{\lambda_1}
\holim_{K\in \mathcal K}\holim_{(K\downarrow p)}\pi_K^*X
\xl{\lambda_2}\holim_{\mathcal K\ltimes F}X.
\]
It remains to see that the equivalence is realized by the map $\lambda$.
 
\medskip
\noindent\textit{Proof of Theorem \ref{holimtheorem}.}
We show that the composition $\lambda_1 \lambda$ is homotopic to $\lambda_2$. 
From this it follows by the above lemmas that $\lambda_1\lambda$ and therefore also 
$\lambda$ is a weak homotopy equivalence. Notice, that an element in the target can be identified with a natural family of maps
\[
\alpha_{(K\xr{f} L,A)}\co B(\mathcal K\downarrow K)\times 
B((K\downarrow p)\downarrow (f,A))\to X(L,A)
\]
indexed by the objects $(K\xr{f} L,A)$ in $(K\downarrow p)$. Let $\beta$ be an element
in $\holim_{\mathcal K\ltimes F}X$ with 
\[
\beta_{(L,A)}\co B(\mathcal K\ltimes F\downarrow (L,A))\to X(L,A)
\] 
for each object $(L,A)$ in $\mathcal K\ltimes F$. 
Then the associated element $\alpha=\lambda_1\lambda(\beta)$ is defined by
\[
\alpha_{(K\xr{f} L,A)}\co B(\mathcal K\downarrow K)\times 
B((K\downarrow p)\downarrow (f,A))\to B(\mathcal K\ltimes F\downarrow (L,A))
\xr{\beta_{(L,A)}}X(L,A)
\]
where the first map is induced by the functor 
\[
\Phi_{(K\xr{f}L,A)} \co (\mathcal K\downarrow K)\times ((K\downarrow p)\downarrow (f,A))\to 
(\mathcal K\ltimes F \downarrow (L,A))  
\]
defined as follows: an object of the domain category is given by the data
\[
K_0\xr{k} K,\quad (K\xr{f_0}L_0,A_0)\xr{(l,a)}(K\xr{f}L,A)
\]
where $k$, $f_0$, and $f$ are morphisms in $\mathcal K$ and $(l,a)\co (L_0,A_0)\to (L,A)$ is a morphism in $\mathcal K\ltimes F$ such that $l f_0=f$. Such an object is mapped by 
$\Phi_{(f,A)}$ to the object 
\[
(f k, F(f_0 k)(a))\co (K_0,F(f_0k)(A_0))\to (L,A).
\] 
A morphism in the domain category is represented by a pair of commutative diagrams of the 
form
\[
\xymatrix{
K_0 \ar[rr]^{k_0} \ar[dr]_{k} & & K_0' \ar[dl]^{k'}\\
& K &
}
 \quad
\xymatrix{
(K\xr{f_0}L_0,A_0) \ar[rr]^{(l_0,a_0)} \ar[dr]_{(l,a)} & & (K\xr{f_0'}L_0',A'_0) \ar[dl]^{(l',a')}\\
& (K\xr{f}L,A) &
}
\]
and this is mapped by $\Phi_{(f,A)}$ to the morphism represented by the diagram
\[
\xymatrix{
(K_0,F(f_0 k)(A_0)) \ar[rr]^{(k_0,F(f_0 k)(a_0))} \ar[dr]_{(f k,F(f_0 k)(a))\ \ \ \ } 
& & (K_0',F(f_0' k')(A_0')) \ar[dl]^{\ \ \ \ \ (f k', F(f_0' k')(a'))}\\
& (L,A). &
}
\]
The element $\lambda_2(\beta)$ is defined analogously using the composite functor
\[ 
\Psi_{(K\xr{f}L,A)} \co (\mathcal K\downarrow K)\times ((K\downarrow p)\downarrow 
(f,A))\to  ((K\downarrow p)\downarrow (f,A))
\to (\mathcal K\ltimes F \downarrow (L,A))
\]
where the first arrow is the projection away from $(\mathcal K\downarrow K)$ and the 
second arrow is induced by $\pi_K$. These functors are related by a natural 
transformation $\Phi\to\Psi$ defined by
\[
\xymatrix{
(K_0,F(f_0 k)(A_0)) \ar[rr]^{(f_0k,\id)} \ar[dr]_{(f k,F(f_0 k)(a))\ \ \ \ } 
& & (L_0,A_0) \ar[dl]^{(l,a)}\\
& (L,A). &
}
\]
This gives rise to a natural homotopy between the induced maps of classifying spaces and thereby to the required homotopy relating $\lambda_1\lambda$ and $\lambda_2$. \qed

\section{Algebraic K-theory of symmetric ring spectra}\label{algebraicKsection}
We recall from \cite{HSS} and \cite{MMSS} that the smash product of symmetric spectra makes $\Sp^{\Sigma}$ a symmetric monoidal category with unit the sphere spectrum $S$. 

\subsection{The spectral category of $A$-modules}\label{spectralsection}
By definition, a symmetric ring spectrum is a monoid in the symmetric monoidal category 
$\Sp^{\Sigma}$. It follows from the universal property of the smash product that a monoid structure on a symmetric spectrum $A$ amounts to a unit $S^0\to A(0)$ and a map of symmetric bispectra
\[
A(m)\wedge A(n)\to A(m+n)
\]
such that the usual diagrams expressing unitality and associativity are commutative; see e.g.\ \cite{Sch}, \cite{Schwede} for details. Similarly, a left module structure of a symmetric ring spectrum $A$ on a symmetric spectrum $E$ amounts to a map of symmetric bi\-spectra
\[
A(m)\wedge E(n)\to E(m+n)
\]
such that the usual module axioms are satisfies. We write $A$-$\mod$ for the topological category of left $A$-modules in which the morphism spaces $\Map_A(E,E')$ are topologized as subspaces of the corresponding morphism spaces $\Map_{\Sp^{\Sigma}}(E,E')$ in 
$\Sp^{\Sigma}$. (Thus, with this definition, $S$-$\mod$ is the same thing as $\Sp^{\Sigma}$).

The symmetric monoidal structure of $\Sp^{\Sigma}$ makes it possible
to talk about \emph{spectral categories}, that is, categories
enriched in symmetric spectra. Such a category $\mathcal C$ is a
class of objects $O\mathcal C$ together with a symmetric spectrum
$\mathcal C(a,b)$ of ``morphisms'' for each pair of objects $a,b$ in
$O\mathcal C$. Furthermore, there is a map of symmetric spectra 
$S\to \mathcal C(a,a)$ for each object $a$ (the unit) and a map of
symmetric spectra
$$
\mathcal C(b,c)\wedge\mathcal C(a,b)\to \mathcal C(a,c)
$$
for each triple of objects $a,b,c$ (the composition). These
structure maps are supposed to satisfy the usual associativity and
unitality axioms for a category. 
A spectral category $\mathcal C$ has an underlying based
topological category with morphism spaces $\mathcal C(a,b)(0)$. 
Given a symmetric ring spectrum $A$, the category $A$-$\mod$ is the underlying
category of a spectral category with morphism spectra denoted $\Hom_A(E,E')$. 
In order to give an explicit description of
the latter, recall the notation $E'[n]$ for the shifted symmetric spectrum from 
Section \ref{symmetricsection}. If $E'$ is an $A$-module then $E'[n]$ inherits an $A$-module
structure defined by
\[
A(h)\wedge E'(n+k)\to E'(h+n+k)\xr{\tau_{h,n}\sqcup 1_k} E'(n+h+k)
\]
and by definition
\[
\Hom_A(E,E')(n)=\Map_A(E,E'[n]).
\]
The structure maps are defined using the $A$-module maps $S^1\wedge E'[n]\to E'[1+n]$ induced by the structure maps of $E'$. We define $\mathcal F_A$ to be the full
subcategory of $A$-$\mod$ containing only the finitely
generated free $A$-modules of the standard form $A^{\vee r}=\bigvee_{i=1}^rA$. For
$n=0$ this is the base object $*$.
The morphism spectra $\Hom_A(A^{\vee r},A^{\vee s})$ in $\mathcal F_A$ may be 
identified with the matrix spectra $M_{s,r}(A)$ from \cite{BHM}, Example 3.2, where
\[
M_{s,r}(A)(n)=\prod_{j=1}^r\bigvee_{i=1}^sA(n).
\]
If we think of this as $s\times r$ matrices with coefficients in $A$
such that each column has at most one non-base point entry, then
composition is given by the usual matrix multiplication.

\subsection{The category $w\mathcal F_A$ of stable equivalences}
\label{stableequivalencessection} 
Let $A$ be a symmetric ring spectrum which we assume to be semistable and well-based in the sense that each of the spaces $A(n)$ has a non-degenerate base point. These are mild conditions on $A$ which allows us to make a simple an explicit construction of the associated algebraic K-theory spectrum $\K(A)$. Most of the symmetric ring spectra that occur in the applications satisfy these conditions and in general any symmetric ring spectrum is stably equivalent to one that is both semistable and well-based.

Recall from \cite{MMSS} that the category $A$-$\mod$ of left $A$-modules has a model category structure in which a map of 
$A$-modules is a weak equivalence if and only if the underlying map of symmetric spectra is a stable equivalence. A fibrant object in this model structure is an $A$-module whose underlying symmetric spectra is an $\Omega$-spectrum. It follows from the discussion in Section  
\ref{spectralsection} that an $A$-module structure on a symmetric spectrum $E$ induces an $A$-module structure on the spectrum $R^{\infty}E$ from Section \ref{symmetricsection}, such that the canonical map $E\to R^{\infty}E$ is a map of $A$-modules. Notice that $A$ being semistable implies that the wedge product $A^{\vee s}$ is semistable as well and that consequently the 
$A$-module $R^{\infty}(A^{\vee s})$ is an $\Omega$-spectrum. Since $A^{\vee r}$ is cofibrant it follows that the mapping spaces
\[
\Map_A(A^{\vee r},R^{\infty}(A^{\vee s}))\cong \prod_{j=1}^rR^{\infty}(A^{\vee s})
\]
represent the ``correct'' homotopy type of the mapping spaces between the objects in 
$\mathcal F_A$. Notice also that an element in this mapping space is a stable equivalence if and only if it induces an isomorphism on $\pi_0$ (the 0th spectrum homotopy group) and that consequently the subspace of stable equivalences is the union of the components that correspond to invertible matrices under the isomorphism
\[
\pi_0\Map_A(A^{\vee r},R^{\infty}(A^{\vee s}))\cong M_{s,r}(\pi_0(A)).
\]
If the ring $\pi_0(A)$ has invariant basis number, then the space of stable equivalences is of course empty unless $r=s$. 
We shall now define a functor $Q_{\I}$ from spectral categories to based topological categories such that when applied to $\mathcal F_A$ we get a topological category $Q_{\I}\mathcal F_A$ whose morphism spaces have the ``correct'' homotopy types described above. 
 Let $\I$ be the category whose objects are the finite sets $\mathbf n=\{1,\dots,n\}$ 
 (including the empty set $\mathbf 0$) and whose morphisms are the injective maps. Given a symmetric spectrum $E$, the sequence of based spaces $\Omega^nE(n)$ defines an  $\I$-diagram in which the morphisms in $\I$ act by conjugation. In detail, if 
 $\alpha\co\mathbf m\to\mathbf n$ is
a morphism in $\mathcal I$, let $\bar\alpha\co \mathbf n=\mathbf
l\sqcup\mathbf m\to\mathbf n$ be the permutation that is order
preserving on the first $l=n-m$ elements and acts as $\alpha$ on the
last $m$ elements. The induced map $\Omega^mE(m)\to\Omega^nE(n)$ takes
an element $f\in \Omega^mE(m)$ to the composition
$$
S^n\xr{\bar\alpha^{-1}}S^l\wedge S^m\xr{S^l\wedge f}S^l\wedge
E(m)\xr{\sigma^l}E(l+m)\xr{\bar\alpha} E(n),
$$
where the $\Sigma_n$-action on $S^n$ is the usual left action. We write $Q_{\mathcal I}(E)$ for the associated based  homotopy colimit
$$
Q_{\mathcal I}(E)=\hocolim_{\mathcal I}\Omega^nE(n),
$$
defined using the topological version of the homotopy colimit functor from \cite{BK}.
This functor is closely related to the functor that maps $E$ to the $0$th space of $R^{\infty}E$.
Indeed, restricting the $\I$-diagram $\mathbf n\mapsto \Omega^nE(n)$ to the subcategory generated by the morphisms $\mathbf n\to \mathbf 1\sqcup\mathbf n$, 
(that is, $i\mapsto 1+i$), we exactly get the diagram defining $R^{\infty}E(0)$.     
\begin{lemma}\label{RQlemma}
If $E$ is semistable, then the canonical map $R^{\infty}E(0)\to Q_{\I}(E)$ is a weak homotopy equivalence.
\end{lemma}
\begin{proof}
It follows from \cite{HSS} and \cite{Sh} that both functors in the lemma take $\pi_*$-~isomorphisms to weak homotopy equivalences. Since $E$ is semistable there exists an 
$\Omega$-spectrum $E'$ and a $\pi_*$-isomorphism $E\to E'$. By naturality it therefore suffices to prove the lemma when $E$ is an $\Omega$-spectrum and in this case the result follows from the fact that the structure maps in the $\I$-diagram defining $Q_{\I}(E)$ are weak homotopy equivalences.
\end{proof}

The advantage of the functor $Q_{\I}\co \Sp^{\Sigma}\to \mathcal T$ is that the symmetric monoidal structure of $\I$ (given by the usual concatenation $\mathbf m\sqcup\mathbf n$ of ordered sets) makes it a (lax) monoidal functor in the sense of 
\cite{MacL}, Section X1.2. Thus, there is a unit $S^0\to Q_{\mathcal I}(S)$ and a natural
multiplication
$$
Q_{\mathcal I}(E)\wedge Q_{\mathcal I}(E')\to Q_{\mathcal I}(E\wedge E')
$$
which is compatible with the coherence
isomorphisms in $\Sp^{\Sigma}$ and $\mathcal T$. The unit is defined by identifying
$S^0$ with $\Omega^0(S^0)$ and including the latter as the $0$th
term in the homotopy colimit. The multiplication is induced by the
natural map of $\mathcal I\times \mathcal I$-diagrams
$$
\Omega^m(E(m))\wedge\Omega^n(E'(n))\to \Omega^{m+n}(E(m)\wedge E'(n))\to
\Omega^{m+n}(E\wedge E'(m+n)),
$$
followed by the map of homotopy colimits induced by the monoidal
structure map $\mathcal I\times \mathcal I\to \mathcal I$. The first
map in the above diagram takes a pair $(f,g)$ to their smash
product. Let now $\mathcal C$ be spectral category as in Section 
\ref{spectralsection} and let $Q_{\I}\mathcal C$ be the based topological category
with the same objects and morphism spaces $Q_{\I}(\mathcal C(a,b))$. The composition 
is induced by the monoidal structure of $Q_{\I}$,
\[
Q_{\I}(\mathcal C(b,c))\wedge Q_{\I}(\mathcal C(a,b))\to 
Q_{\I}(\mathcal C(b,c)\wedge\mathcal C(a,b))\to Q_{\mathcal I}(\mathcal C(a,c) 
\]
and the units are defined by $S^0\to Q_{\I}(S)\to Q_{\I}(\mathcal C(a,a))$. Applying this to the spectral category $\mathcal F_A$ we get the topological category $Q_{\I}\mathcal F_A$ with morphism spaces
\[
Q_{\I}\mathcal F_A(A^{\vee r},A^{\vee s})=Q_{\I}(\Hom_A(A^{\vee r},A^{\vee s})).
\]
Using Lemma \ref{RQlemma} and the fact that $R^{\infty}$ preserves products of semistable 
symmetric spectra up to level equivalence we get a chain of weak homotopy equivalences
\[
Q_{\I}\mathcal F_A(A^{\vee r},A^{\vee s})\xl{\sim} R^{\infty}\Hom_A(A^{\vee r},A^{\vee s})(0)
\xr{\sim}\Map_A(A^{\vee r},R^{\infty}(A^{\vee s}))
\]
which shows that the morphism spaces in $Q_{\I}\mathcal F_A$ have the desired homotopy types. 

\begin{definition}
The category $w \mathcal F_A$ of stable equivalences in $\mathcal
F_A$ is the topological subcategory of $Q_{\mathcal I}\mathcal F_A$
that has the same objects as the latter and in which the morphism
spaces
$$
w\mathcal F_A(A^{\vee r},A^{\vee s})\subseteq 
Q_{\mathcal I}\mathcal F_A(A^{\vee r},A^{\vee s})
$$
are the unions of those components that correspond to invertible matrices in
\[
\pi_0Q_{\mathcal I}\mathcal F_A(A^{\vee r},A^{\vee s})\cong M_{s,r}(\pi_0(A)).
\]
\end{definition}

\subsection{Algebraic K-theory of $\mathcal F_A$}\label{K-theorysection}
Let again $A$ be a symmetric ring spectrum that is well-based and semistable.
We now make the extra assumption that $A$ be connective in the sense that the spectrum homotopy groups vanish in negative degrees. The connectivity assumption is in fact not 
needed for any of the constructions in this section, but we do not claim that our definition of algebraic K-theory is the ``correct'' definition if $A$ is not connective. 

Our construction is based on Segal's $\Gamma$-space
approach \cite{S} to infinite loop spaces which we briefly recall first. 
Let $\Gamma^{\op}$ be the category of finite based sets. Following \cite{BF}, a
$\Gamma$-space is a functor $M\co \Gamma^{\op}\to \mathcal T$
such that $M(*)=*$. Given based sets $X$ and $Y$, the last condition ensures that there is a
natural transformation $X\wedge M(Y)\to M(X\wedge Y)$, defined by applying $M$ to the
based map $y\mapsto (x,y)$ for each $x$ in $X$. Let
$S^n_{\bullet}$ be the $n$-fold smash product of the standard
simplicial circle $\Delta_{\bullet}[1]/\partial
\Delta_{\bullet}[1]$, and let $M(S^n)$ be the realization of the
simplicial space $M(S^n_{\bullet})$ obtained by applying $M$ to
$S^n_{\bullet}$ in each simplicial degree. In the following we shall tacitly assume that the simplicial spaces $M(S^n_{\bullet})$ are \emph{good} in the sense that the degeneracy maps are cofibrations, see \cite{S}, Appendix A. This ensures that the realization is homotopically 
well-behaved. The symmetric spectrum $M(S)$ associated to $M$ has $n$th space $M(S^n)$ and structure maps induced by the simplicial maps
$$
S^1_{\bullet}\wedge M(S^n_{\bullet})\to M(S^1_{\bullet}\wedge
S^n_{\bullet}) =M(S^{n+1}_{\bullet}).
$$
We say that $M$ is \emph{special} if, for each pair of finite based 
sets $X$ and $Y$, the natural map $M(X\vee Y)\to M(X)\times
M(Y)$ is a weak homotopy  equivalence. In this case it follows from \cite{S},
Proposition 1.4, that $M(S)$ is a positive $\Omega$-spectrum. 
We say that $M$ is \emph{very special} if $M(S)$ is
a genuine $\Omega$-spectrum. This is equivalent to $M(S^0)$ being a
grouplike monoid.

Since $\mathcal F_A$ is a category with finite coproducts, Segal's
construction in \cite{S} applies to give a $\Gamma$-category with 
$\mathcal F_A$ as its underlying category. We shall
consider a variant of this where, roughly
speaking, instead of using all sum diagrams of objects in
$\mathcal F_A$, we only consider those that arise from
permutation matrices. Let $\mathcal F$ be the skeleton category
of finite sets with objects $\mathbf n=\{1,\dots, n\}$, including the empty set $\mathbf 0$.
This is a category with coproducts, hence it gives rise to a $\Gamma$-category
$X\mapsto \mathcal F\langle X\rangle$. It is useful to formulate
this construction in terms of sum diagrams in $\mathcal F$. Given
a finite based set $X$, let $\bar X$ be the unbased set obtained
by excluding the base point, and let $\mathcal P(\bar X)$ denote
the category of subsets and inclusions in $\bar X$. An object in
$\mathcal F\langle X\rangle$ may then be identified with a
functor $\theta\co \mathcal P(\bar X)\to\mathcal F$ that takes
disjoint unions to coproducts: if $U$ and $V$ are
disjoint subsets of $\bar X$, then the diagram
$\theta_U\to\theta_{U\cup V}\leftarrow \theta_V$ represents the
middle term as a coproduct in $\mathcal F$. A morphism in
$\mathcal F\langle X\rangle $ is a natural transformation 
of such functors. We now enrich this construction to a topological 
$\Gamma$-category $X\mapsto \mathcal F_A\langle X\rangle$ such that 
the objects of $\mathcal F_A\langle X\rangle$ are the functors 
$A[\theta]\co \mathcal P(\bar X)\to \mathcal F_A$ of the special form 
$A[\theta]=A\wedge \theta_+$ for an object $\theta$ in 
$\mathcal F\langle X\rangle$. The morphism spaces  
$\Map_A(A[\theta],A[\theta'])$ are the spaces 
of natural transformations 
between such functors, equipped with the subspace topology induced from
the product of the mapping spaces $\Map_A(A[\theta_U],A[\theta'_U])$. 
We extend the definition of $\mathcal F_A\langle X\rangle$ to a spectral 
category with morphism spectra defined by
$$
\Hom_A(A[\theta],A[\theta'])(n)=\Map_A(A[\theta],A[n][\theta'])
$$
where the shifted symmetric spectrum $A[n]$ is defined as in Section 
\ref{symmetricsection}. This is equivalent to defining $\Hom_A(A[\theta],A[\theta'])$ as the 
\emph{end} of the $\mathcal P(\bar X)^{\op}\times \mathcal P(\bar X)$-diagram 
$\Hom_A(A[\theta_U],A[\theta'_V])$, see \cite{MacL}, Section IX.5.  
In this way we get a special $\Gamma$-spectral category in the
sense that there are natural equivalences of spectral categories
$$
\mathcal F_A\langle X\vee Y\rangle\xr{\sim}\mathcal F_A\langle
X\rangle \times \mathcal F_A\langle Y\rangle.
$$
The morphism spectra of the spectral category 
$\mathcal F_A\langle X\rangle \times \mathcal F_A\langle Y\rangle$ are the 
products of the morphism spectra of $\mathcal F_A\langle X\rangle$ and 
$\mathcal F_A\langle Y\rangle$. It follows that there are isomorphisms of morphism 
spectra
\[
\Hom_A(A[\theta],A[\theta'])\cong\prod_{x\in \bar X}
\Hom_A(A[\theta_{\{x\}}],A[\theta'_{\{x\}}]).
\]
We now proceed as in Section \ref{stableequivalencessection} and use the
monoidal functor $Q_{\mathcal I}$ to define the  topological $\Gamma$-category
$X\mapsto Q_{\mathcal I}\mathcal F_A\langle X\rangle$ and the
$\Gamma$-subcategory of stable equivalences $X\mapsto w\mathcal
F_A\langle X\rangle$.
Applying the usual classifying space functor we get
the $\Gamma$-space $B(w\mathcal F_A)$. Thus,
$B(w\mathcal F_A)\langle X\rangle$ is the realization of the
simplicial space
with $k$-simplices
$$
\coprod_{\theta_0,\dots,\theta_k}w\mathcal F_A\langle X\rangle (A[\theta_1],A[\theta_0)])
\times\dots\times w\mathcal F_A\langle X\rangle(A[\theta_k],A[\theta_{k-1}]),
$$
where $\theta_0,\dots,\theta_k$ runs through all $(k+1)$-tuples of objects in
$\mathcal F\langle X\rangle$.

\begin{lemma}\label{Kspecialgammalemma}
The $\Gamma$-space $B(w\mathcal F_A)$ is special.
\end{lemma}
\begin{proof}
Given finite based sets $X$ and $Y$, we must prove that the functor
$$
w\mathcal F_A\langle X\vee Y\rangle\to w\mathcal F_A\langle X\rangle \times w\mathcal F_A\langle Y\rangle
$$
induces an equivalence of classifying spaces. Since $Q_{\mathcal
I}$ commutes with products only up to equivalence and not
isomorphism, this is not quite an equivalence of categories. Let
$$
w(\mathcal F_A\langle X\rangle\times \mathcal F_A\langle
Y\rangle)\subseteq Q_{\mathcal I}(\mathcal F_A\langle X\rangle
\times\mathcal F_A\langle Y\rangle)
$$
be the category of stable equivalences associated to the spectral product
category $\mathcal F_A\langle X\rangle \times\mathcal F_A\langle
Y\rangle$. The above functor may then be factorized as
$$
w\mathcal F_A\langle X\vee Y\rangle\to
w(\mathcal F_A\langle X\rangle\times \mathcal F_A\langle
Y\rangle)\to w\mathcal F_A\langle
X\rangle \times w\mathcal F_A\langle Y\rangle.
$$
Here the first functor is an equivalence of categories since the spectral 
category $\mathcal F_A$ is special. The second functor is the identity
on objects and induces an equivalence on morphism spaces, hence it induces
a degree-wise equivalence of the simplicial spaces defining the bar 
constructions. It follows from the assumption that $A$ be well-based that 
these are good simplicial spaces and the topological realization is therefore also a 
weak homotopy equivalence.
\end{proof}

\begin{definition}
The algebraic K-theory spectrum associated to $w\mathcal F_A$ is
the spectrum
$$
\K(A)=B(w\mathcal F_A)\langle S\rangle.
$$
\end{definition}
It follows from Lemma \ref{Kspecialgammalemma} that this is a positive
$\Omega$-spectrum.

\section{Cyclic algebraic K-theory}\label{cyclicalgebraicKsection}
In this section we consider the cyclic algebraic K-theory spectrum
$\K^{\cy}(A)$. This is a spectrum with cyclotomic structure and we begin by
a general discussion of epicyclic and cyclotomic structures.

\subsection{Epicyclic and cyclotomic
  structures}\label{cyclotomicsection}
We first recall the edgewise subdivision functors from \cite{BHM},
Section 1. Let $\Delta$ be the simplicial category viewed as a
monoidal category under the usual ordered concatenation of ordered sets.
For each positive integer $r$, the $r$-fold concatenation functor
$\sqcup_r\co \Delta\to \Delta$ is defined by
$$
\sqcup_r[k]=\underbrace{[k]\sqcup\dots\sqcup[k]}_r=[r(k+1)-1].
$$
Given a simplicial space $X_{\bullet}$, viewed as a
contravariant functor on $\Delta$, the $r$-fold edgewise subdivision
is the composition $\sd_r X=X\circ\sqcup_r$. Let
$\Delta[k]$ denote the standard $k$-simplex,
$$
\Delta[k]=\{(t_0,\dots,t_k)\in I^{k+1}\co t_0+\dots+t_k=1\}.
$$
The correspondence $[k]\mapsto \Delta[k]$ defines a cosimplicial
space in the usual way and there is a cosimplicial map
$$
D_r\co \Delta[k]\to\Delta[\sqcup_r[k]],\quad v\mapsto
(\textstyle\frac{1}{r}v,\dots,\frac{1}{r}v)
$$
for each $r$. It follows from \cite{BHM}, Lemma 1.1, that the induced map
\begin{equation}\label{D_r}
D_r\co |\sd_rX_{\bullet}|\to|X_{\bullet}|,\quad [x,v]\mapsto [x,D_rv],
\end{equation}
is a homeomorphism for any simplicial space $X_{\bullet}$.
Suppose now that $X_{\bullet}$ is a cyclic space with cyclic operators
$t_k$ acting on $X_k$. By \cite{BHM}, Lemma 1.8 and Lemma 1.11,
$|\sd_rX_{\bullet}|$ and $|X_{\bullet}|$ then come equipped with
actions of the circle group $\mathbb T$ and $D_r$ is $\mathbb
T$-equivariant. Furthermore, $\sd_rX_{\bullet}$ inherits a simplicial action
of the cyclic group $C_r$ by letting the preferred generator act on
$\sd_rX_k$ by
$$
t_{r(k+1)-1}^{k+1}\co X_{r(k+1)-1}\to X_{r(k+1)-1}
$$
and the induced $C_r$ action on the realization agrees with that
obtained by restricting the $\mathbb T$-action. Notice that the
cyclic structure of $X_{\bullet}$ restricts to a cyclic structure on
the fixed points $\sd_rX_{\bullet}^{C_r}$.

\begin{definition}[Goodwillie \cite{G}]
An epicyclic space is a cyclic space $X_{\bullet}$ equipped with a
family of cyclic maps $R_r\co\sd_rX_{\bullet}^{C_r}\to X_{\bullet}$
for $r\geq1$, such that (i) $R_1=\id$ and (ii) the diagrams
$$
\begin{CD}
\sd_r(\sd_sX_{\bullet}^{C_s})^{C_r}@= \sd_{rs}X^{C_{rs}}_{\bullet}\\
@VV \sd_r R_s^{C_r} V @VV R_{rs}V\\
\sd_rX_{\bullet}^{C_r}@> R_r>> X_{\bullet}
\end{CD}
$$
are commutative. An epicyclic spectrum is a cyclic spectrum equipped
with a family of cyclic spectrum maps $R_r$ as above satisfying (i)
and (ii) in each spectrum degree.
\end{definition}

We shall also need the following $\mathbb T$-equivariant analogue. 
Let $\rho_r\co \mathbb T\to\mathbb T/C_r$ be
the homeomorphism $\rho_r(z)=\sqrt[r]z$. Given a $\mathbb T$-space $X$,
we denote by $\rho_r^*X^{C_r}$ the $\mathbb T$-space obtained by pulling
back the $\mathbb T/C_r$-action on $X^{C_r}$ via $\rho_r$.
\begin{definition}
A cyclotomic space is a $\mathbb T$-space $X$ equipped with a
family of $\mathbb T$-equivariant maps $R_r\co\rho^*_rX^{C_r}\to
X$ for $r\geq1$, such that (i) $R_1=\id$ and (ii) the diagrams
$$
\begin{CD}
\rho^*_r(\rho^*_sX^{C_s})^{C_r}@= \rho^*_{rs}X^{C_{rs}}\\
@VV \rho^*_rR_s^{C_r} V @VV R_{rs}V\\
\rho_r^*X^{C_r}@> R_r>> X
\end{CD}
$$
are commutative. A spectrum with cyclotomic structure is a spectrum with $\mathbb T$-action and a family of $\mathbb T$-equivariant maps $R_r$ as above satisfying (i) and (ii) in each spectrum degree. 
\end{definition}

\begin{remark}
A spectrum with cyclotomic structure is not the same as a cyclotomic spectrum in the sense of \cite{HM}. The difference is analogous to the distinction between a spectrum with 
$\mathbb T$-action and a genuine $\mathbb T$-spectrum, see e.g.\ \cite{C}. Thus, forgetting part of the structure, a cyclotomic spectrum as in \cite{HM} 
gives a spectrum with cyclotomic structure in our sense.
\end{remark}

An epicyclic structure on a cyclic space or spectrum $X_{\bullet}$
induces a cyclotomic structure on the topological realization
$|X_{\bullet}|$. This uses that the homeomorphisms $D_r$
restrict to $\mathbb T$-equivariant homeomorphisms
$$
|\sd_rX_{\bullet}^{C_r}|\to\rho_r^*|\sd_rX_{\bullet}|^{C_r}\xr{D_r^{C_r}}
\rho^*_r|X_{\bullet}|^{C_r}
$$
when the domain is given the $\mathbb T$-action induced by the cyclic
structure  of $\sd_rX_{\bullet}^{C_r}$.
The cyclotomic structure maps are then defined by
$$
R_r\co \rho^*_r|X_{\bullet}|^{C_r}\cong
|\sd_rX_{\bullet}^{C_r}|\to |X_{\bullet}|
$$
where the last map is the realization of the epicyclic
structure map. 

Recall the category $\mathcal N$ from Example \ref{Iexample} with objects the natural numbers and morphisms $r\co m\to n$ a natural number $r$ such that $m=nr$. A cyclotomic space $X$ gives rise to an $\mathcal N$-diagram of $\mathbb T$-spaces $n\mapsto \rho_n^*X^{C_n}$ in which the structure maps of the diagram are defined by
\begin{equation}\label{Rstructure}
\rho_m^*X^{C_m}=\rho_n^*(\rho_r^*X^{C_r})^{C_n}
\xr{\rho_n^* R_r^{C_n}}\rho_n^*X^{C_n}.
\end{equation}
We write $RX$ for the associated homotopy limit, 
\begin{equation}\label{Rfunctor}
RX=\holim_{R_r}\rho_n^*X^{C_n}
\end{equation}
(there should be no risk of confusion with the functor $R$ on symmetric spectra considered in Section \ref{symmetricsection}). The cyclotomic structure of $X$ induces a cyclotomic structure on each of the spaces $\rho_n^*X^{C_n}$ such that the structure maps in the diagram are cyclotomic. It follows that the correspondence $X\mapsto RX$ defines an endofunctor on the category of cyclotomic spaces. It is worth noting that $R$ has the structure of a comonad 
(see \cite{MacL}, Section VI.1) with counit $RX\to X$ defined by restricting to the terminal object $1$ in $\mathcal N$ and comultiplication $RX\to RRX$ induced by the functor 
$\mathcal N\times \mathcal N\to\mathcal N$ that takes $(r,s)$ to $rs$.

Let $\mathbb I\ltimes \mathbb T$ be the category introduced as a Grothendieck construction in Example~ \ref{ITexample}. In general, an $\mathbb I\ltimes\mathbb T$-diagram $X$ amounts to a sequence of $\mathbb T$-spaces $X(n)$ for $n\geq 1$, equipped with two families of structure maps
\[
F_r, R_r\co X(nr)\to X(n),
\]
such that the relations (\ref{Icategory}) in the category $\mathbb I$ hold, the maps $R_r$ are 
$\mathbb T$-equivariant, and $F_r(z^rx)=zF_r(x)$ for all $z\in \mathbb T$ and $x\in X(nr)$.
A cyclotomic space $X$ defines an $\mathbb I\ltimes\mathbb T$-diagram 
$n\mapsto \rho_n^*X^{C_n}$ by letting
\[
F_r\co \rho_{rn}^*X^{C_{rn}}\xr{}
\rho^*_nX^{C_n}\]
be the natural subspace inclusion and $R_r$ the map defined in (\ref{Rstructure}).
It follows from Example \ref{ITexample} that there are natural weak homotopy equivalences
\[
\holim_{\mathbb I\ltimes \mathbb T}\rho_n^*X^{C_n}\xr{\sim}
(\holim_{R_r}\rho_n^*X^{C_n})^{h(N\ltimes \mathbb T)}\xr{\sim}
(\holim_{R_r}(\rho_n^*X^{C_n})^{h\mathbb T})^{hN},
\]
where in fact the first map is a homeomorphism. 
\subsection{The cyclic bar construction}\label{cyclicbarsection}
\label{B^cysection} Waldhausen's cyclic bar construction
$B^{\cy}_{\bullet}(\mathcal C)$ of a small topological category
$\mathcal C$ provides a basic example of an epicyclic space. The underlying cyclic 
space has the form
\[
B^{\cy}_{\bullet}(\mathcal C)\co [k]\mapsto\coprod_{c_0,\dots,c_k}
\mathcal C(c_0,c_k)\times\mathcal C(c_1,c_0)\times\dots\times
\mathcal C(c_k,c_{k-1}),
\]
where the coproduct is over all $(k+1)$-tuples of objects in
$\mathcal C$ and the structure maps are of the usual Hochschild
type, see \cite{BHM}, Section 2. A $k$-simplex in
$\sd_rB^{\cy}_{\bullet}(\mathcal C)$ may be represented as a tuple
of morphisms of the form
\begin{equation}\label{sdrep}
\{f_i(j)\co 0\leq i\leq k,\quad 1\leq j\leq r\},
\end{equation}
such that if $\zeta_r$ denotes the generator of $C_r$, then the
$C_r$-action induced by the cyclic structure is given by
$$
\zeta_r\cdot \{f_i(j)\}=\{f'_i(j)\},\quad \text{where } f'_i(j)=
\begin{cases}
f_i(r),&\text{for $j=1$},\\
f_i(j-1),&\text{for $1< j\leq r$}.
\end{cases}
$$
It follows that a $C_r$-fixed point $\{f_i(j)\}$ is constant in the
$j$-coordinate such that the diagonal inclusion defines an isomorphism
\[
\Delta_r\co B^{\cy}_{\bullet}(\mathcal C)\to 
\sd_rB^{\cy}_{\bullet}(\mathcal C)^{C_r}
\]
of cyclic spaces. The epicyclic structure maps are the inverse 
isomorphisms
$$
R_r\co \sd_rB^{\cy}_{\bullet}(\mathcal C)^{C_r}\to
B^{\cy}_{\bullet}(\mathcal C).
$$
Let us write $B^{\cy}(\mathcal C)$ for the realization of 
$B^{\cy}_{\bullet}(\mathcal C)$.  
Since the cyclotomic structure maps of $B^{\cy}(\mathcal C)$ are homeomorphisms, 
it follows that the canonical map from the
categorical limit to the homotopy limit induces a weak homotopy equivalence
\begin{equation}\label{B^cyRequivalence}
B^{\cy}(\mathcal C)\cong\lim_{R_r}\rho^*_nB^{\cy}(\mathcal C)^{C_n}
\to \holim_{R_r}\rho^*_nB^{\cy}(\mathcal C)^{C_n}.
\end{equation}
This is compatible with the actions of the Frobenius operators if we
define the action on $B^{\cy}(\mathcal C)$ by
\[
F_r\co |B^{\cy}_{\bullet}(\mathcal C)|\xr{|\Delta_r|}
|\sd_rB^{\cy}_{\bullet}(\mathcal C)^{C_r}|\xr{} 
|\sd_rB^{\cy}_{\bullet}(\mathcal C)|\xr{D_r}
|B^{\cy}_{\bullet}(\mathcal C)|
\]
where the second map is the inclusion. It is clear from the definition that this fits together with the $\mathbb T$-action to give an $N\ltimes\mathbb T$-action on $B^{\cy}(\mathcal C)$. 
This is the action consider in Theorem \ref{Bcytheorem}.

\subsection{Cyclic algebraic K-theory}\label{cyclicK-theorysection}
Let now $A$ be a connective symmetric ring spectra which is semistable and well-based. Applying the cyclic bar
construction to the $\Gamma$-category $w\mathcal F_A$ introduced in
Section \ref{stableequivalencessection}, we get
the $\Gamma$-space $B^{\cy}(w\mathcal F_A)$. This is a
$\Gamma$-cyclotomic space, that is, a $\Gamma$-object in the category
of cyclotomic spaces.
\begin{definition}
The cyclic algebraic K-theory spectrum of $w\mathcal F_A$ is the
spectrum with cyclotomic structure
$$
\K^{\cy}(A)=B^{\cy}(w\mathcal F_A)\langle S\rangle.
$$
\end{definition}
In general it is not true that a natural transformation of
functors induces a homotopy after applying the cyclic bar
construction. However, as follows from the discussion in 
\cite{DM}, Section 1.5, this is the case for natural \emph{isomorphisms} and in particular 
equivalent categories have equivalent cyclic classifying spaces. Using this, the proof of the
following lemma is analogous to that of Lemma
\ref{Kspecialgammalemma}.
\begin{lemma}
The $\Gamma$-space $B^{\cy}(w\mathcal F_A)$ is special and
$\K^{\cy}(A)$ is a positive $\Omega$-spectrum. \qed
\end{lemma}

We now apply Theorem \ref{Bcytheorem} to the groupoid-like 
$\Gamma$-category $w\mathcal F_A$ and consider the diagram of $\Gamma$-spaces
\[
B^{\cy}(w\mathcal F_A)^{h(N\ltimes\mathbb T)}\xr{\sim} \Map(BN,B(w\mathcal F_A))\xl{}
B(w\mathcal F_A)
\]
where the right hand map is induced by the projection $BN\to *$. We define the $\Gamma$-space $B'(w\mathcal F_A)$ to be the homotopy pullback of this diagram and we write $\K'(A)$ for the associated spectrum. It follows from the definition that $\K'(A)$ is canonically equivalent to 
$\K(A)$ and that there is a natural diagram
\[
\K(A)\xl{\sim} \K'(A)\xr{} K^{\cy}(A)^{h(N\ltimes\mathbb T)}.
\]
The right hand map is in spectrum degree $n$ given by the composition
\[
|B'(w\mathcal F_A\langle S^n_{\bullet}\rangle)|
\to |B(w\mathcal F_A\langle S^n_{\bullet}\rangle)^{h(N\ltimes \mathbb T)}|
\to |B(w\mathcal F_A\langle S^n_{\bullet}\rangle)|^{h(N\ltimes \mathbb T)}.
\]
\section{The cyclotomic trace}\label{cyclotomictracesection}
In this section $A$ denotes a connective symmetric ring spectrum which we 
as usual assume to be semistable and well-based. Applying a construction 
analogous to that of Dundas-McCarthy \cite{D1}, \cite{DM},
we define the topological cyclic homology of the spectral category $\mathcal F_A$ 
and we construct the cyclotomic trace using this model. 
We compare our definitions to the models of topological cyclic homology 
considered  by Goodwillie \cite{G} and Hesselholt-Madsen \cite{HM} at 
the end of the section. 
\subsection{Topological cyclic homology}\label{THsection}
We first introduce some convenient notation. Let
$Q_{\I^{k+1}}$ be the functor that to a $(k+1)$-fold
multi-symmetric spectrum $E$ associates the based homotopy colimit
\[
Q_{\I^{k+1}}(E)=\hocolim_{\I^{k+1}} \Map(S^{n_0}\wedge\dots\wedge
S^{n_k},E(n_0,\dots,n_k)).
\]
The structure maps of the $\I^{k+1}$-diagram on the right hand side
are similar to those for $Q_{\I}$. We define the spectrum homotopy groups of a $(k+1)$-fold multi-symmetric spectrum $E$ by
\[
\pi_i(E)=\colim_{n_0,\dots,n_k}\pi_{i+n_0+\dots+n_k}(E(n_0,\dots,n_k))
\]
and we say that a map of multi-symmetric spectra is a $\pi_*$-isomorphism if it induces an isomorphism on spectrum homotopy groups.  
 Using that homotopy colimits over $\I^{k+1}$ can be 
calculated iteratively, the following lemma follows from an easy inductive argument based on Lemma \ref{RQlemma}.
\begin{lemma}\label{Qmulti-lemma}
Let $E\to E'$ be a $\pi_*$-isomorphism of $(k+1)$-fold multi-symmetric spectra that are semistable in each spectrum variable (keeping the remaining spectrum variables fixed). Then the induced map 
\[
Q_{\I^{k+1}}(E)\to Q_{\I^{k+1}}(E')
\]
is a weak homotopy equivalence.\qed 
\end{lemma}

Given a family of symmetric spectra 
$E_0,\dots, E_k$, we write $E_0\bar\wedge\dots\bar\wedge E_k$ for the 
$(k+1)$-fold multi-symmetric spectrum defined by
\[
E_0\bar\wedge\dots\bar\wedge E_k(n_0,\dots,n_k)
=E_0(n_0)\wedge\dots\wedge E_k(n_k).
\]
Let $\mathcal C$ be a spectral category as in 
Section \ref{spectralsection} and suppose that $\mathcal C$ is small in the sense that the set of objects form a set. For each $k\geq 0$ we define a $(k+1)$-fold multi-symmetric spectrum 
$V_k[\mathcal C]$ by
\[
V_k[\mathcal C]=\bigvee_{c_0,\dots,c_k}\mathcal C(c_0,c_k)\bar\wedge
\mathcal C(c_1,c_0)\bar\wedge\dots\bar\wedge 
\mathcal C(c_k,c_{k-1})
\]
where the wedge product is over all $(k+1)$-tuples of objects in $\mathcal C$. It is clear 
from the definition that letting $\mathcal C$ vary we get a functor 
$V_k[-]$ from small spectral categories to 
$(k+1)$-multi-symmetric spectra. Applying this to the $\Gamma$-category 
$X\mapsto \mathcal F_A\langle X\rangle $ from Section \ref{K-theorysection} we therefore get 
a $\Gamma$-object $X\mapsto V_k[\mathcal F_A\langle X\rangle]$ in the category of 
$(k+1)$-multi-symmetric spectra. Explicitly, with notation as in Section \ref{K-theorysection},
\[
V_k[\mathcal F_A\langle X\rangle]=\bigvee_{\theta_0,\dots,\theta_k}
\Hom_A(A[\theta_0],A[\theta_k]))\bar\wedge\dots
\bar\wedge\Hom_A(A[\theta_k],A[\theta_{k-1}]),
\]
where $\theta_0,\dots,\theta_k$ runs through all $(k+1)$-tuples of objects
in $\mathcal F\langle X\rangle$.
We define $X\mapsto \TH(\mathcal F_A\langle X\rangle)$ to be
the realization of the $\Gamma$-epicyclic space that to a based set
$X$ associates the epicyclic space
$$
\TH_{\bullet}(\mathcal F_A\langle X\rangle)\co
[k]\mapsto Q_{\I^{k+1}}(V_k[\mathcal F_A\langle X\rangle]).
$$
The cyclic structure is defined as in \cite{DM}, Section 1.3, and the
epicyclic structure maps
$$
R_r\co \sd_r\TH_{\bullet}(\mathcal F_A\langle X\rangle)^{C_r}\to
\TH_{\bullet}(\mathcal F_A\langle X\rangle)
$$
are defined as in \cite{DM}, Section 1.5.
\begin{definition}
The topological Hochschild homology spectrum $\TH(\mathcal F_A)$
is the realization of the associated epicyclic spectrum
$\TH_{\bullet}(\mathcal F_A\langle S\rangle)$.
\end{definition}
It follows from the discussion in Section \ref{cyclotomicsection} that $\TH(\mathcal F_A)$ is a spectrum with cyclotomic structure.
With notation as in that section, let $\TR(\mathcal F_A)$ be the homotopy limit of the 
$\mathcal N$-diagram $n\mapsto \rho_n^*\TH(\mathcal F_A)^{C_n}$ defined by the restriction maps $R_r$.  
\begin{definition}
The topological cyclic homology spectrum $\TTC(\mathcal F_A)$ is defined by
\[
\TTC(\mathcal F_A)=\holim_{\mathbb I\ltimes \mathbb T}\rho_n^*\TH(\mathcal F_A)^{C_n}
=\TR(\mathcal F_A)^{h(N\ltimes\mathbb T)}.
\]
\end{definition}
The other variants of topological cyclic homology are defined analogously using the subcategories in (\ref{subcategories}).  
We show that $\TH(\mathcal F_A)$ as well as the fixed point
spectra $\TH(\mathcal F_A)^{C_n}$ are $\Omega$-spectra in
Proposition \ref{omegabispectrumproposition} below. It follows that also
$\TR(\mathcal F_A)$, $\TTC(\mathcal F_A)$, and the other variants of topological cyclic homology are $\Omega$-spectra.

\subsection{The cyclotomic trace}\label{tracesection}
The construction of the cyclotomic trace is based on a map of
$\Gamma$-epicyclic spaces
$$
B^{\cy}_{\bullet}(w\mathcal F_A\langle X\rangle)
\to\TH_{\bullet}(\mathcal F_A\langle X\rangle).
$$
Recall that the space of $k$-simplices in 
$B^{\cy}_{\bullet}(w\mathcal F_A)\langle X\rangle$ is defined by
$$
\coprod_{\theta_0,\dots,\theta_k}
w\mathcal F_A(A[\theta_0],A[\theta_k]))\times\dots \times 
w\mathcal F_A(A[\theta_k],A[\theta_{k-1}]),
$$
where $\theta_0,\dots,\theta_k$ runs through all $(k+1)$-tuples of objects in 
$\mathcal F\langle X\rangle$. The restriction of the above map
to the component indexed by a fixed $(k+1)$-tuple
$\theta_0,\dots,\theta_k$  is defined by the composition
\begin{align*}
&w\mathcal F_A(A[\theta_0],A[\theta_k])\times \dots\times 
w\mathcal F_A(A[\theta_k],A[\theta_{k-1}])\\
&\to Q_{\I}(\Hom_A(A[\theta_0],A[\theta_k]))\wedge\dots\wedge Q_{\mathcal I}
(\Hom_A(A[\theta_k],A[\theta_{k-1}]))\\
&\to Q_{\mathcal I^{k+1}}
(\Hom_A(A[\theta_0],A[\theta_k])\bar\wedge\dots\bar\wedge
\Hom_A(A[\theta_k],A[\theta_{k-1}]))\\
&\to Q_{\mathcal I^{k+1}}(V_k[\mathcal F_A\langle X\rangle]).
\end{align*}
Here the first map is the inclusion of the stable equivalences in the full morphism 
spaces, the second map is the map of homotopy colimits induced by the
natural transformation that takes a $(k+1)$-tuple of maps to their
smash product, and the last map is induced by the inclusion of the wedge
summand in $V_k[\mathcal F_A\langle X\rangle]$ indexed by $\theta_0,\dots,\theta_k$. 
There results a map of epicyclic spectra and, after topological
realization, a map of spectra with cyclotomic structure
$
\K^{\cy}(A)\to\TH(\mathcal F_A).
$
Passing to the homotopy limits over the restriction maps and composing with  
the equivalence induced by (\ref{B^cyRequivalence}), we get a map
of spectra with cyclotomic structure
$$
\K^{\cy}(A)\stackrel{\sim}{\to}\holim_{R_r}\rho_n^*\K^{\cy}(A)^{C_n}\to 
\holim_{R_r}\rho_n^*\TH(\mathcal F_A)^{C_n}=\TR(\mathcal F_A).
$$
The cyclotomic trace is obtained from this by evaluating the $N\ltimes \mathbb T$-homotopy fixed points and composing with the chain of maps constructed in Section 
\ref{cyclicK-theorysection}. 

\begin{definition}
The cyclotomic trace is the chain of maps
represented by the following diagram of spectra
$$
\trc\co\K(A)\xl{\sim}\K'(A)\xr{}\K^{\cy}(A)^{h(N\ltimes\mathbb T)}\to \TR(\mathcal F_A)^{h(N\ltimes\mathbb T)}=\TTC(\mathcal F_A).
$$
\end{definition}
\subsection{The topological cyclic homology spectrum $\TTC(A)$}
\label{HMsection}
Following Dundas-McCarthy, we now relate the above constructions to
the models $\TTC(A)$ and $\TC(A)$ of topological cyclic homology considered 
by Goodwillie \cite{G} and 
Hesselholt-Madsen \cite{HM}. These versions are based on B\"okstedt's definition 
\cite{B}, \cite{Sh} of the topological Hoch\-schild homology spectrum $\TH(A)$. The latter is the
realization of the epicyclic symmetric spectrum whose $n$th space is given by
$$
\TH_{\bullet}(A,n)\co[k]\mapsto Q_{\I^{k+1}}(S^n\wedge
A^{\bar\wedge(k+1)}).
$$
Here we view $A^{\bar\wedge(k+1)}$ as a $(k+1)$-fold multi-symmetric
spectrum in the usual way and the epicyclic structure maps are defined
as for $\TH_{\bullet}(\mathcal F_A)$. Writing $\TR(A)$ for the
homotopy limit of the fixed point spectra $\rho_n^*\TH(A)^{C_n}$ over the restriction maps, 
the spectrum $\TTC(A)$ is defined by
\[
\TTC(A)
=\holim_{\mathbb I\ltimes\mathbb T}\rho_n^*\TH(A)^{C_n}
=\TR(A)^{h(N \ltimes\mathbb T)}.
\]
In order to relate this definition to that based on the cyclotomic
spectrum $\TH(\mathcal F_A)$, we extend the definition of the
latter to give a bispectrum. Consider for each $n$ the 
$\Gamma$-epicyclic space
$$
\TH_{\bullet}(\mathcal F_A\langle-\rangle ,n)\co
(X,[k])\mapsto Q_{\I^{k+1}}(S^n\wedge V_k
[\mathcal F_A\langle X\rangle] )
$$
with structure maps similar to those for $\TH(\mathcal F_A)$.
Evaluating this $\Gamma$-space on the sphere spectrum $S$ in the usual way
we get a bispectrum, also denoted $\TH(\mathcal F_A)$, that in 
bidegree $(m,n)$ takes the value 
$\TH(\mathcal F_A\langle S^m\rangle,n)$.

\begin{proposition}\label{omegabispectrumproposition}
The bispectrum $\TH(\mathcal F_A)$ and its fixed point spectra
$\TH(\mathcal F_A)^{C_r}$ are $\Omega$-bispectra.
\end{proposition}
\begin{proof}
The condition that $A$ be semistable implies that it is $\pi_*$-isomorphic to a 
symmetric ring spectrum which is an $\Omega$-spectrum. Thus, using Lemma 
\ref{Qmulti-lemma}, we may assume without loss of generality that $A$ is an $\Omega$-spectrum. The connectivity assumption on $A$ then implies that the $n$th space $A(n)$ is $(n-1)$-connected and  that the structure maps $S^1\wedge A(n)\to A(n+1)$ are $2n$-connected. It follows that the $\I^{k+1}$-diagram giving rise to $\TH_{k}(\mathcal F_A\langle X\rangle,n)$ satisfies the connectivity assumptions required for B\"okstedt's approximation lemma for homotopy colimits, see \cite{M}, Lemma 2.3.7.
Using the notation above, we first keep $m$ fixed and claim that the
adjoint structure maps in the $n$-variable are equivalences. By
definition, the $(m,n)$th space is the realization
of the bisimplicial space $\TH_{\bullet}(\mathcal F_A\langle
S^m_{\bullet}\rangle,n)$ and the adjoint structure maps are defined by
the compositions
$$
|\TH_{\bullet}(\mathcal F_A\langle S^m_{\bullet}\rangle,n)|\to
|\Omega\TH_{\bullet}(\mathcal F_A\langle S^m_{\bullet}\rangle,n+1)|
\to
\Omega|\TH_{\bullet}(\mathcal F_A\langle S^m_{\bullet}\rangle,n+1)|.
$$
It follows from B\"okstedt's approximation lemma and the 
connectivity assumptions on $A$ that the adjoint structure maps
$$
\TH_{\bullet}(\mathcal F_A\langle S^m_{\bullet}\rangle,n)\to
\Omega\TH_{\bullet}(\mathcal F_A\langle S^m_{\bullet}\rangle,n+1)
$$
are equivalences in each bidegree and the realization is therefore
also a weak homotopy equivalence. Since $\Omega$ commutes 
with realization up to
equivalence for good simplicial connected spaces by \cite{May},
Theorem 12.3, the second map is a weak homotopy equivalence as well. 
In order to get the same conclusions for the fixed point spectra we
use the edgewise subdivision functor as in the proof of \cite{DM},
Lemma 1.6.11, and apply a similar argument. Next we keep $n$ fixed
and claim that the $\Gamma$-spaces $\TH(\mathcal F_A\langle -\rangle,n)^{C_r}$ are
special in the sense that the composite map
\[
\xymatrix{
\TH(\mathcal F_A\langle X\vee Y\rangle,n)^{C_r}
\ar[r]\ar[dr] &  
\TH(\mathcal F_A\langle X\rangle\times \mathcal F_A\langle
Y\rangle,n)^{C_r} \ar[d] \\
& \TH(\mathcal F_A\langle X\rangle,n)^{C_r}\times
\TH(\mathcal F_A\langle Y\rangle,n)^{C_r} 
}
\]
is a weak homotopy equivalence for each pair of finite based sets $X$ and $Y$. 
Here the middle term denotes the effect of applying the construction from Section 
\ref{THsection} to the spectral category 
$\mathcal F_A\langle X\rangle\times \mathcal F_A\langle Y\rangle$. 
Since the horizontal map is induced by an equivalence of spectral categories it is 
a weak homotopy equivalence by \cite{DM}, Proposition
1.6.6, and the vertical map is a weak homotopy equivalence since topological Hochschild 
homology preserves products of spectral categories up to equivalence by  
\cite{DM}, Proposition 1.6.15. It follows from the 
first part of the proof that  $\TH(\mathcal F_A\langle -\rangle,n)^{C_r}$
is in fact a very special $\Gamma$-space and the associated
spectrum is therefore an $\Omega$-spectrum.
\end{proof}

We write $\TH(\mathcal F_A)$ and $\TH'(\mathcal F_A)$ for the two symmetric spectra
obtained respectively by restricting to bidegrees $(n,0)$ and $(0,n)$. Thus, 
$\TH(\mathcal F_A)$ retains its meaning from Section \ref{THsection}.
Letting $\mathfrak C$ denote the family of finite cyclic subgroups of $\mathbb T$, we say that a map of ($\Omega$-)spectra with $\mathbb T$-action is a
 \emph{$\mathfrak C$-equivalence} if the induced maps of fixed point spectra are equivalences for all finite cyclic subgroups.  
It follows from Proposition \ref{omegabispectrumproposition} that
$\TH(\mathcal F_A)$ and $\TH'(\mathcal F_A)$ are related by the following explicit 
chain of $\mathfrak C$-equivalences
\begin{align*}
\TH(\mathcal F_A\langle S^n\rangle,0)&\stackrel{\sim}{\to}
\hocolim_{l,m}\Omega^{l+m}\TH(\mathcal F_A\langle S^{l+n}\rangle,m)\\
&\stackrel{\sim}{\leftarrow}
\hocolim_{l,m}\Omega^{l+m}(S^n\wedge\TH(\mathcal F_A\langle S^l\rangle,m))\\
&\stackrel{\sim}{\to}
\hocolim_{l,m}\Omega^{l+m}\TH(\mathcal F_A\langle S^l\rangle,n+m)
\stackrel{\sim}{\leftarrow}
\TH(\mathcal F_A\langle S^0\rangle,n).
\end{align*}
Passing to homotopy limits over the associated fixed point spectra
we therefore get a chain of level equivalences
\[
\TTC(\mathcal F_A)\simeq\TTC'(\mathcal F_A).
\] 
Let now $\mathcal F_A(1)$ be the full spectral 
subcategory of $\mathcal F_A$ containing only the rank 1 module $A$ itself. Identifying $\TH(A)$ with $\TH(\mathcal F_A(1))$ we get a canonical map $\TH(A)\to \TH'(\mathcal F_A)$ of spectra with cyclotomic structure. 

\begin{proposition}[\cite{DM}]\label{Moritaprop}
The map $\TH(A)\to\TH'(\mathcal F_A)$ is a level-wise $\mathfrak
C$-equivalence, hence gives rise to a level-wise equivalence
$\TTC(A)\stackrel{\sim}{\to}\TTC'(\mathcal F_A)$ and similarly for the other variants of topological cyclic homology.
\end{proposition}
\begin{proof}
Let $\mathcal F_A(n)$ be the full spectral subcategory of $\mathcal F_A$ containing the free 
$A$-modules $A^{\vee r}$ with $r\leq n$, and let $\TH'(\mathcal F_A(n))$ be the realization   
of the associated epicyclic spectrum
$$
\TH'_{\bullet}(\mathcal F_A(n),m)\co[k]\mapsto
Q_{\I^{k+1}}(S^m\wedge V_k[\mathcal F_A(n)]).
$$
The inclusions $\mathcal F_A(n)\to
\mathcal F_A$ give rise to a $\mathfrak C$-equivalence
$$
\hocolim_n\TH'(\mathcal F_A(n))\stackrel{\sim}{\to}
\TH'(\mathcal F_A),
$$
hence it suffices to show that the inclusion of $\mathcal F_A(1)$ in
$\mathcal F_A(n)$ induces a $\mathfrak C$-equivalence for all $n$.
Writing $M_n(A)$ for the symmetric ring spectrum $\Hom_A(A^{\vee
  n},A^{\vee n})$, it follows from 
the proof of $\TH$-cofinality in \cite{DM}, Lemma 2.1.1 that the natural map
$$
\TH(M_n(A))\to \TH'(\mathcal F_A(n))
$$
is a $\mathfrak C$-equivalence. In order to specify a homotopy
inverse, consider the map
$$
\TH'_{\bullet}(\mathcal F_A(n))\to\TH_{\bullet}(M_n(A))
$$
obtained by extending a map defined on some summands of $A^{\vee
  n}$ to the whole module by collapsing the remaining summands.
This is a precyclic map in the sense that it commutes
with the face and cyclic operators but not with the degeneracy
operators, see \cite{DM}, Section 1.5, for details.
Consider then the commutative
diagram of precyclic maps
$$
\begin{CD}
\TH_{\bullet}(A)@=\TH_{\bullet}(A)\\
@VVV @VVV\\
\TH_{\bullet}'(\mathcal F_A(n))@>>>\TH_{\bullet}(M_n(A)).
\end{CD}
$$
The vertical map on the right hand side is a $\mathfrak C$-equivalence
by \cite{DM}, Proposition 1.6.18, hence the left hand map is also a
$\mathfrak C$-equivalence as claimed.
\end{proof}

Combining the above results we get the following corollary.

\begin{corollary}
There is a chain of level equivalences
 \[
\TTC(\mathcal F_A)\simeq \TTC'(\mathcal F_A)\xl{\sim}\TTC(A)
\]
and similarly for the other variants of topological cyclic homology. 
\end{corollary}

\section{Homotopy fixed points of the cyclic bar construction}\label{hofixsection}
In this section we analyze the homotopy fixed points of the cyclic bar construction 
$B^{\cy}(\mathcal C)$ under the $N\ltimes\mathbb T$-action introduced in Section 
\ref{B^cysection}. The main point is to prove Theorem \ref{Bcytheorem} which characterizes these homotopy fixed points in terms of the mapping space $\Map(BN,B(\mathcal C))$ when 
$\mathcal C$ is groupoid-like. In general, a left $N\ltimes\mathbb T$-action on a space $X$ amounts to a $\mathbb T$-action together with a family of maps $F_r\co X\to X$ for $r\geq 1$, such that $F_1$ is the identity, $F_r\circ F_s=F_{rs}$, and
$F_r(z^rx)=zF_r(x)$ for all $z$ in $\mathbb T$ and $x$ in $X$. It follows from the discussion in Example \ref{semidirectexample} that $N$ acts on the homotopy fixed points $X^{h\mathbb T}$ and that there is a natural weak homotopy equivalence
\[
X^{h(N\ltimes \mathbb T)}\xr{\sim}(X^{h\mathbb T})^{hN}.
\]
We now specialize to the cyclic bar construction $B^{\cy}(\mathcal C)$ associated to a topological category $\mathcal C$. By definition, the homotopy fixed points is the space of $\mathbb T$-equivariant maps $\Map_{\mathbb T}
(E\mathbb T,B^{\cy}(\mathcal C))$ where $E\mathbb T$ denotes the one-sided bar construction 
$B(\mathbb T,\mathbb T,*)$. Consider the  composite map
\begin{equation}\label{Bcymap}
\Map_{\mathbb T}(E\mathbb T,B^{\cy}(\mathcal C))\to B^{\cy}(\mathcal C)\to B(\mathcal C)
\end{equation}
where the first map is defined by evaluating a function $\alpha\co E\mathbb T\to 
B^{\cy}(\mathcal C)$ at the base point of $E\mathbb T$ (determined by the unit of $\mathbb T$)
and, in the notation from Section \ref{B^cysection}, the second map is the projection that in simplicial degree $k$ forgets the morphism from $c_0$ to $c_k$.  

\begin{lemma}\label{Bcylemma}
If $\mathcal C$ is groupoid-like, then the map $B^{\cy}(\mathcal C)^{h\mathbb T}\to 
B(\mathcal C)$ defined in (\ref{Bcymap}) is a weak homotopy equivalence. 
\end{lemma}
\begin{proof}
We first consider the composition 
\[
\mathbb T\times B^{\cy}(\mathcal C)\to B^{\cy}(\mathcal C)\to B(\mathcal C)
\]
where the first map is the $\mathbb T$-action on $B^{\cy}(\mathcal C)$ and the second is the projection considered above. The adjoint is a $\mathbb T$-equivariant map to the free loop space of $B(\mathcal C)$, 
\[
B^{\cy}(\mathcal C)\to \Map(\mathbb T,B(\mathcal C)),
\]  
and it is well-known that this is a weak homotopy equivalence when $\mathcal C$ is 
groupoid-like. The argument is similar to that used to prove that the cyclic bar construction of a (well-based) grouplike topological monoid $G$ is equivalent to the free loop space on $BG$, see e.g.\ \cite{G1}. As a technical point, our assumption that the topological category $\mathcal C$ be well-based (see Section \ref{holimDefsection}) implies that the simplicial spaces 
$B_{\bullet}(\mathcal C)$ and $B^{\cy}_{\bullet}(\mathcal C)$ are good in the sense of \cite{S}, Appendix A. It follows that the map of $\mathbb T$-homotopy fixed points induced by the above map is also a weak homotopy equivalence. 
The homotopy fixed points of the free loop space are determined by
\[
\Map(\mathbb T,B(\mathcal C))^{h\mathbb T}\simeq\Map_{\mathbb T}
(\mathbb T\times E\mathbb T,B(\mathcal C))\simeq\Map(E\mathbb T,B(\mathcal C))\xr{\sim}B(\mathcal C)
\]
where the last map is defined by evaluating a function at the base point of $E\mathbb T$. It follows easily from the definition that the composition of the weak homotopy equivalences
\[
B^{\cy}(\mathcal C)^{h\mathbb T}\xr{\sim}\Map(\mathbb T,B(\mathcal C))^{h\mathbb T}\xr{\sim}B(\mathcal C) 
\]
is the map claimed to be a weak homotopy equivalence in the lemma. 
\end{proof}

In the following we shall view $B(\mathcal C)$ as an $N$-space with trivial action. The map in Lemma \ref{Bcylemma} is then not strictly compatible with the $N$-actions, but we shall prove that it is so up to canonical coherent homotopies which is enough to get a natural map of homotopy fixed points. We first introduce some machinery which is convenient for analyzing homotopy fixed points of $N$-actions. 

\subsection{Homotopy fixed points for $N$-actions}
Consider in general an $N$-space $X$. 
Writing $\mathcal P$ for the set of prime numbers, we 
identify $N$ with the free commutative monoid generated by $\mathcal
P$ and we shall view $X$ as a space equipped with a family of
commuting operators
\begin{equation}\label{commutingoperatorseq}
F_p\co X\to X,\quad p\in \mathcal P.
\end{equation}
By definition, the homotopy fixed points of $X$ are defined by
$$
X^{hN}=\Map_{N}(B(N,N,*),X),
$$
where $B(N,N,*)$ denotes the one-sided bar construction and the
right hand side is the space of $N$-equivariant maps. It is easy to see 
that if $EN$ is any contractible free
$N$-CW complex, then $X^{hN}$ is homotopy equivalent to $\Map_N(EN,X)$. In
the following we shall consider a model $EN$ that is convenient
for writing down explicit homotopies. Given a finite subset
$U\subseteq \mathcal P$, let $\langle U\rangle$ be the submonoid
of $N$ generated by $U$. We let
$$
E\langle U\rangle =\prod_{p\in U} [0,\infty),
$$
and give this the product action of $\langle U\rangle $
in which an element $p\in U$ acts on the $p$th component by
translation, $t_p\mapsto t_p+1$. Notice that there is a canonical
inclusion of $E\langle U\rangle$ in the 1-skeleton of $B(N,N,*)$ and
that this induces an $N$-equivariant homotopy equivalence.
Given a $\langle U\rangle $-space $X$, we now redefine the
homotopy fixed points by
$$
X^{h\langle U\rangle}=\Map_{\langle U\rangle}(E\langle U\rangle, X).
$$
We shall need some notation for such homotopy fixed points. Let $I^U$
be the $|U|$-dimensional unit cube with coordinates indexed by the
elements of $U$. Given a subset $V\subseteq U$, we define the
\emph{$V$\!th lower face} of $I^U$ to be the $|U-V|$-dimensional cube
$$
\partial_VI^U=\{(t_p)\in I^U\co t_p=0 \text{ for }p\in V\}.
$$
Similarly, we define the \emph{$V$\!th upper face} of $I^U$ by
$$
\partial^V\!I^U=\{(t_p)\in I^U\co t_p=1 \text{ for }p\in V\}.
$$
We shall often identify $\partial_VI^U$ and $\partial^V\!I^U$ with
$I^{U-V}$ in the canonical way. For a map $\alpha\co I^U\to X$ we
define
$$
\partial_V\alpha,\ \partial^V\!\alpha\co I^{U-V}\to X
$$
by respectively restricting to $\partial_VI^U$ and $\partial^V\!I^U$. 
Suppose now that $X$ is a space with a $\langle U\rangle$-action
specified by a family of commuting operators $F_p$
as in (\ref{commutingoperatorseq}). Given $V\subseteq U$, we write $F_V$
for the composition of the $F_p$'s indexed by $p\in V$. With this
notation we may identify $X^{h\langle U\rangle}$ with the subspace of
the mapping space $\Map(I^U,X)$ defined by the condition that
$$
\partial^V\!\alpha=F_V\circ\partial_V\alpha, \quad \text{for all $V\subseteq U$}.
$$
We let $EN$ be the colimit of the spaces $E\langle U\rangle$ under the
natural inclusions (using the point $0$ in $[0,\infty)$ as vertex) and redefine the 
homotopy fixed points of an $N$-space by
$$
X^{hN}=\Map_N(EN,X).
$$

\subsection{Coherent homotopies}
Given $N$-spaces $X$ and $Y$ we shall now make explicit what it
means for a map $f\co X\to Y$ to be compatible with the actions up
to coherent homotopy. Let us first consider the situation in which
a pair of spaces $X$ and $Y$ each comes equipped with a self-map,
denoted respectively by $F^X$ and $F^Y$. In this case the
condition for a map $f$ to be homotopy compatible with the
actions is simply that there exists a homotopy $h\co X\times I\to
Y$ from $f\circ F^X$ to $F^Y\circ f$. A choice of such a homotopy determines a
map of homotopy fixed points by concatenating $f\alpha$ and
$h(\alpha(0),-)$, that is,
$$
f^{h}\co X^{hF^X}\to Y^{hF^Y},\quad
f^{h}(\alpha)(t)=
\begin{cases}
f\alpha(2t),& 0\leq t\leq 1/2,\\
h(\alpha(0),2t-1),&1/2\leq t\leq 1.
\end{cases}
$$
\begin{lemma}\label{hofixlemma}
If $f$ is a weak homotopy equivalence, then so is $f^{h}$.
\end{lemma}
\begin{proof}
We identify the homotopy fixed points of $F^X$ with the pullback of
the diagram
$$
X\xr{(\id_X,F^X)}X\times X \xl{(\ev_0,\ev_1)} X^I
$$
and, letting
$$
\bar Y=\{(y,\omega)\in Y\times Y^I\co \omega(1)=F^Y(y)\}
$$
and rescaling, we identify the homotopy fixed points of $F^Y$ with the
pullback of the diagram
$$
\bar Y\xr{(p_Y,\ev_0)} Y\times Y\xl{(\ev_0,\ev_1)} Y^I.
$$
Here $\ev_0$ and $\ev_1$ evaluate a path at its endpoints and
$(p_Y,\ev_0)$ is defined by $(y,\omega)\mapsto (y,\omega(0))$.
From this point of view, $f^{h}$ is induced by a map of pullback
diagrams which is a term-wise weak homotopy equivalence. The 
result now follows from the fact that these diagrams are homotopy 
cartesian. 
\end{proof}
Let us now return to the case of two $N$-spaces $X$ and $Y$ and
let us write $F_p^X$ and $F_p^Y$ for the corresponding operators
(\ref{commutingoperatorseq}). Then we say that a map $f\co X\to Y$
is compatible with the actions up to coherent homotopy if there is
a family of higher homotopies
$$
h^U\co X\times I^U\to Y,
$$
indexed on the finite subsets $U\subseteq \mathcal P$,
such that $h^{\emptyset}=f$ and
\begin{equation}\label{Ncoherenceeq}
\partial_Vh^U=h^{U-V}\circ(F_V^X\times I^{U-V}), \quad
\partial^V\!h^U=F^Y_V\circ h^{U-V},
\end{equation}
whenever $V\subseteq U$. Here $\partial_Vh^U$ and
$\partial^Vh^U$ are the maps $X\times I^{U-V}\to Y$ obtained by
restricting to $\partial_VI^U$ and $\partial^VI^U$.

\begin{proposition}\label{Ncoherenceproposition}
Let $X$ and $Y$ be $N$-spaces and let $f\co X\to Y$ be a map that is
compatible with the actions up to coherent homotopy in the above
sense. Then a choice of coherent homotopies determines a map
$$
f^{h}\co X^{hN}\to Y^{hN}
$$
and if $f$ is a weak homotopy equivalence, then so is $f^{h}$.
\end{proposition}
\begin{proof}
By definition, $X^{hN}$ is the limit of the tower of fibrations
defined by the homotopy fixed points $X^{h\langle U\rangle}$ and
similarly for $Y^{hN}$. Thus, it suffices to construct a
compatible family of maps
$$
f^{h\langle U\rangle}\co X^{h\langle U\rangle}\to X^{h\langle U\rangle}
$$
such that if $f$ is an equivalence, then so is $f^{h\langle U\rangle}$
for each $U$. By compatible we mean that the diagrams
\[
\begin{CD}
X^{h\langle U\rangle}@>f^{h\langle U\rangle}>> Y^{h \langle U\rangle}\\
@VV \partial_{U-V} V @VV \partial_{U-V} V\\
X^{h\langle V\rangle}@>f^{h\langle V\rangle}>> Y^{h \langle V\rangle}
\end{CD}
\]
commute whenever $V\subseteq U$. In order to define these maps we 
subdivide $I^U$ in $|U|^2$ subcubes by introducing a new vertex at 
the midpoint of each edge. For each subset $V\subseteq U$, let $I^U_V$ 
be the subcube
$$
I^U_V=\Big\{(t_p)\in I^U\co
\begin{cases}
0\leq t_p\leq 1/2,& p\notin V\\
1/2\leq t_p\leq 1,&p\in V
\end{cases}
\Big\}.
$$
Given an element $\alpha$ in $X^{h\langle U\rangle}$, we shall define
$f^{h\langle U\rangle}\alpha$ by specifying its restriction to each of these
subcubes. For each $V\subseteq U$, consider the composite map
$h^V\partial_V\alpha$ defined by
$$
I^U\simeq I^{U-V}\times I^V\simeq \partial_V I^U\times
I^V\xr{\partial_V\alpha\times I^V}X\times I^V\xr{h^V}Y,
$$
where the first map permutes the coordinates. Identifying $I^U$ with
$I^U_V$ via the canonical coordinate-wise affine
homeomorphism, this defines the restriction of $f^{h\langle
  U\rangle}\alpha$ to $I^U_V$.  It follows from the
definition of a coherent homotopy that this is a well-defined
element in $Y^{h\langle U\rangle}$. Furthermore, given disjoint sets
$U$ and $U'$,
this construction is compatible with the canonical isomorphism
$$
\langle U\cup U'\rangle\simeq \langle U\rangle \times \langle
U'\rangle
$$
in the sense that there is a commutative diagram
$$
\begin{CD}
X^{h\langle U\cup U'\rangle}@>f^{h\langle U\cup U'\rangle}>>
Y^{h\langle U\cup U'\rangle} \\
@VV\simeq V @VV\simeq V\\
\big(X^{h\langle U\rangle }\big)^{h\langle U'\rangle}@>\big(f^{h\langle
U\rangle}\big)^{h\langle U'\rangle}>>
\big(Y^{h\langle U\rangle }\big)^{h\langle U'\rangle}.
\end{CD}
$$
Using this together with Lemma \ref{hofixlemma}, it follows by induction
that if $f$ is an equivalence, then so is $f^{h\langle U\rangle}$.
\end{proof}

\subsection{The proof of Theorem \ref{Bcytheorem}}\label{coherenceproofsect} 
In order to finish the proof of Theorem \ref{Bcytheorem} we must show that the projection 
$B^{\cy}(\mathcal C)\to B(\mathcal C)$ is compatible with the $N$-actions up to coherent homotopy when we give $B(\mathcal C)$ the trivial action. For this purpose 
we introduce a new family of operators $\bar F_r$ on
$B^{\cy}(\mathcal C)$. Let $\bar D_r\co \id_{\Delta}\to\sqcup_r$ be
the natural transformation that includes $[k]$ as the last
component in $\sqcup_r[k]$ and use the same notation for the
associated map of cosimplicial spaces,
$$
\bar D_r\co\Delta[k]\to\Delta[\sqcup_r[k]].
$$
Notice that there is a cosimplicial homotopy
$$
\Delta[k]\times I\to \Delta[\sqcup_r[k]],\quad (v,t)\mapsto
(1-t)D_rv+t\bar D_rv
$$
relating this to the map $D_r$ from Section
\ref{cyclotomicsection}. If $X_{\bullet}$ is a simplicial space
we get an induced map
$$
\bar D_r\co|\sd_rX_{\bullet}|\to|X_{\bullet}|,\quad [x,v]\mapsto
     [x,\bar D_rv]
$$
that is homotopic to the homeomorphism (\ref{D_r}) by the above
homotopy. It follows from the definition that $\bar D_r$ is the
topological realization of the simplicial map $\bar
D_r^*\co\sd_r X_{\bullet}\to X_{\bullet}$ defined by
$$
\bar D^*_{r,k}=d_0^{(r-1)(k+1)}\co \sd_r X_k=X_{r(k+1)-1}\to X_k.
$$
The definition of the operator $\bar F_r$ is now analogous to the definition
of $F_r$ in Section~\ref{cyclicbarsection} except that we use $\bar D_r$ instead of $D_r$,
$$
\bar F_r\co  |B_{\bullet}^{\cy}(\mathcal C)|
\xr{|\Delta_r|} |\sd_rB^{\cy}_{\bullet}(\mathcal C)^{C_r}|
\xr{}|\sd_rB^{\cy}_{\bullet}(\mathcal C)|
\xr{\bar D_r} |B^{\cy}_{\bullet}(\mathcal C)|.
$$
One checks that $\bar F_r\bar
F_s=\bar F_{rs}$ such that these operators define an $N$-action on
$B^{\cy}(\mathcal C)$. The following lemma states that the identity  
on $B^{\cy}(\mathcal C)$ is compatible with these two $N$-actions 
up to coherent homotopy. 
\begin{lemma}\label{FbarFlemma}
The $N$-actions on $B^{\cy}(\mathcal C)$ induced by the $F_r$ and the
$\bar F_r$ operators are compatible up to coherent homotopies.
\end{lemma}
\begin{proof}
We must produce a higher homotopy
$$
h^U\co B^{\cy}(\mathcal C)\times I^U\to B^{\cy}(\mathcal C)
$$
for each finite subset $U\subseteq \mathcal P$, such that
$h^{\emptyset}$ is the identity and the relations in
(\ref{Ncoherenceeq}) are satisfied, that is,
\[
\partial_Vh^U=h^{U-V}\circ (F_V\times I^{U-V}),\qquad 
\partial^Vh^U=\bar F_V\circ h^{U-V},
\]
whenever $V\subseteq U$. Let $\sqcup_U\co \Delta\to \Delta$ be the composition 
of the concatenation functors $\sqcup_p$ for $p\in U$ and  
consider the homotopies
$$
h^U\co \Delta[k]\times I^U\to \Delta[\sqcup_U[k]]
$$
defined by
$$
h^U(v,(t_p))=\prod_{p\in U}((1-t_p)D_pv+t_p\bar D_pv).
$$
Here we use the notation
$$
D_pv\cdot D_{q}v=D_{pq}v,\quad \bar D_pv\cdot\bar D_{q}v=
\bar D_{pq}v
$$
and make the convention that when both $D_p$ and $\bar D_q$ occur in a
product, then we apply $D_p$ first, that is,
$$
D_pv\cdot \bar D_{q}v=(\bar D_{q}\circ D_p)v,\quad \bar
D_{p}v\cdot D_{q}v=(\bar D_{p}\circ D_{q})v.
$$
Thus, for example,
\begin{align*}
h^{\{p,q\}}(v,(t_p,t_{q}))=
&(1-t_p)(1-t_{q})D_{pq}v+(1-t_p)t_{q}(\bar
D_{q}\circ D_{p})v\\
&+t_p(1-t_{q})(\bar D_p\circ
D_{q})v+t_pt_{q}\bar D_{pq}v.
\end{align*}
Then, with notation as in (\ref{Ncoherenceeq}), we have the relations
\begin{equation}\label{Dhomotopieseq}
\partial_Vh^U=h^{U-V}\circ(D_V\times
I^{U-V}),\quad\partial^Vh^U=\bar D_V\circ h^{U-V}.
\end{equation}
If we view $I^U$ as a constant cosimplicial space, then
$h^U$ defines a map of cosimplicial spaces, hence induces a natural
map
$$
h^U\co|\sd_UX_{\bullet}|\times I^U\to |X_{\bullet}|
$$
for any simplicial space $X_{\bullet}$. Here $\sd_U$
denotes the composition of the functors $\sd_p$ for $p\in U$. Applying
this to $B^{\cy}_{\bullet}(\mathcal C)$ and writing $C_U$ for the
cyclic group of order the product of the elements in $U$,
the requested homotopies are defined by
$$
h^U\co |B^{\cy}_{\bullet}(\mathcal C)|\times I^U
\xr{\Delta_U}
|\sd_UB^{\cy}_{\bullet}(\mathcal C)^{C_U}|\times I^U\to
|\sd_UB^{\cy}_{\bullet}(\mathcal C)|\times
I^U\xr{h^U}
|B^{\cy}_{\bullet}(\mathcal C)|.
$$
For $U=\emptyset$, we define $C_{\emptyset}$ to be the trivial group and
$h^{\emptyset}$ to be the identity on $B^{\cy}(\mathcal C)$.
It follows from (\ref{Dhomotopieseq}) that these homotopies satisfy
the required coherence relations.
\end{proof}

\begin{corollary}\label{coherentcorollary}
The projection $B^{\cy}(\mathcal C)\to B(\mathcal C)$ is compatible with the $N$-actions up to coherent homotopy.  
\end{corollary}
\begin{proof}
It follows immediately from the definition that if $B^{\cy}(\mathcal C)$ is equipped with the $N$-action induced by the $\bar F_r$ operators, then the projection is $N$-equivariant when we give $B(\mathcal C)$ the trivial action. The result therefore follows from Lemma \ref{FbarFlemma}.
\end{proof}

\medskip
\noindent\textit{Proof of Theorem \ref{Bcytheorem}.}
Using Lemma \ref{Bcylemma} and Proposition \ref{Ncoherenceproposition} it suffices to show that the map in (\ref{Bcymap}) is compatible with the $N$-actions up to coherent homotopy when we give $B(\mathcal C)$ the trivial action. It is clear from the definition that the first map in  
(\ref{Bcymap}) is $N$-equivariant since the base point in $E\mathbb T$ is fixed by the 
$N$-action. The result therefore follows from Corollary \ref{coherentcorollary}. \qed

\end{document}